\documentclass[11pt]{article}
\usepackage{amssymb}
\usepackage{amsmath}
\usepackage{mathrsfs}
\usepackage{graphics}
\usepackage{graphicx}
\usepackage{xcolor}
\usepackage{subfigure}
\usepackage[T1]{fontenc}
\usepackage{latexsym,amssymb,amsmath,amsfonts,amsthm}\usepackage{txfonts}
\newcommand{\be}{\begin{eqnarray}}
\newcommand{\ben}{\begin{eqnarray*}}
\newcommand{\en}{\end{eqnarray}}
\newcommand{\enn}{\end{eqnarray*}}

\newtheorem{theorem}{Theorem}[section]
\newtheorem{lemma}{Lemma}[section]
\newtheorem{prp}[theorem]{Proposition}
\newtheorem{thm}[theorem]{Theorem}

\newtheorem{dfn}{Definition}[section]

\newtheorem{remark}{Remark}

\definecolor{rr}{rgb}{0,0,0}
\begin{document}
\renewcommand{\theequation}{\arabic{section}.\arabic{equation}}
\begin{titlepage}
\title{\bf  On the small time asymptotics of quasilinear parabolic stochastic partial differential equations
}
\author{ Rangrang Zhang\\
{\small  School of  Mathematics and Statistics,
Beijing Institute of Technology, Beijing, 100081, China.}\\
({\sf rrzhang@amss.ac.cn})}
\date{}
\end{titlepage}
\maketitle

\noindent\textbf{Abstract}:
In this paper, we establish a small time large deviation principles for the quasilinear parabolic stochastic partial differential equations with multiplicative noise, which are neither monotone nor locally monotone.



\section{Introduction}
In this paper, we are interested in the small time asymptotics of quasilinear parabolic partial differential equations, which can be written as
\begin{eqnarray}\label{equ-0}
\left\{
  \begin{array}{ll}
   du+div(B(u))dt=div(A(u)\nabla u)dt+\sigma(t,u)dW(t),
  & \ x\in \mathbb{T}^d,\  t\in [0,1] \\
   u(0)=f. &
  \end{array}
\right.
\end{eqnarray}
Where $W(t), t\geq 0$ is a cylindrical Brownian motion, $A(\cdot), B(\cdot)$ are appropriate coefficients specified later. The solution of (\ref{equ-0}) is denoted by $u=u(x,t)$. The precise description of the problem will be presented in the next section.

The quasilinear parabolic partial differential equations can model the phenomenon of convection-diffusion of the ideal fluids and therefore arise in a wide variety of important applications, including for instance two or three phase flows in porus media or sedimentation-consolidation processes (see, e.g. \cite{GMT} and the references therein ). The addition of a stochastic noise to this physical model is fully natural as it represents external random perturbations or a lack of knowledge of certain physical parameters.

There are several recent works about the existence and uniqueness of pathwise weak solution of the above equation, i.e. strong in the probabilistic sense and weak in the PDE sense. We mention
two of them which are relevant to our work. Debussche et al.\cite{DHV} obtained the existence and uniqueness of the Cauchy problem for quasilinear degenerate parabolic stochastic partial differential equations by a Yamada-Watanabe type argument and kinetic formulation.
    For the nondegenerate case, Hofmanov\'{a} and Zhang \cite{H-Z} put forward a direct
    (and therefore much simpler) approach to establish the global well-posedness of the pathwise
     solution. Based on the results of \cite{H-Z}, Dong et al. \cite{DZZ} established the Freidlin-Wentzell's type large deviation principles for the strong solution of the quasilinear nondegenerate parabolic stochastic partial differential equations.

 The purpose of this paper is to establish the small time asymptotics for the quasilinear  nondegenerate parabolic stochastic partial differential equations, which describes the behavior of the solution at a very small time. Specifically, we focus on the limiting behavior of the solution of the quasilinear parabolic stochastic partial differential equations in a time interval $[0,t]$ as $t$ goes to zero. An important motivation to consider such problem comes from Varadhan identity
\begin{eqnarray*}
\lim_{t\rightarrow 0}2t\log P\big(u(0)\in B,\ u(t)\in C\big)=-d^2(B,C),
\end{eqnarray*}
where $u$ is the strong solution of the quasilinear parabolic stochastic partial differential equations  and $d$ is an appropriate Riemann distance associated with the diffusion generated by $u$.
It's worth mentioning that the mathematical study of the small time asymptotics for finite dimensional processes was initiated by Varadhan \cite{V}. For the infinite dimensional diffusion processes, several works studied them, for example, \cite{A-K, A-Z, F-Z, H-R, ZTS} and the references therein.

\

Up to now, there are some results of the small time asymptotics for infinite dimensional stochastic partial differential equations. For example, Xu and Zhang \cite{X-Z} established the small time asymptotics of 2D Navier-Stokes equations in the state space $C([0,T],H)$. Dong and Zhang \cite{D-Z} proved the small time asymptotics of 3D stochastic primitive equations in the state space  $C([0,T],H^1)$. In this article, we are interested in the small time asymptotics of the quasilinear parabolic stochastic partial differential equations. The key step is to prove the solution $u^{\varepsilon}(t)=u(\varepsilon t)$ of (\ref{equ-0}) is exponentially equivalent to the law of the solution of
\begin{eqnarray*}
v^\varepsilon(t)=f+\sqrt{\varepsilon}\int^t_0 \sigma(\varepsilon s, v^\varepsilon(s))dW(s).
\end{eqnarray*}
The exponential equivalence will be achieved through several approximations.
The hard part is to deal with the nonlinear term $div(A(u)\nabla u)$ of (\ref{equ-0})
because it is neither monotone nor locally monotone.
To overcome this difficulty,
we adopt the method from \cite{H-Z} to introduce the heat kernel $\{P_r\}_{r>0}$ to smooth the operator $A$.
Meanwhile, in order to obtain some higher Sobolev norm estimates of $u^{\varepsilon}(t)$,
we impose a stronger condition on the initial value $f$. Thereby, under the stronger condition on $f$,
we establish the exponential equivalence by making use of properties of $\{P_r\}_{r>0}$ and
higher regularity of $u^{\varepsilon}(t)$.
In the rest part, by using suitable approximation and analytical techniques, we succeed to remove the stronger condition on $f$.

 \

This paper is organized as follows. The mathematical formulation of quasilinear parabolic stochastic partial differential equations is in Section 2. In Section 3, we introduce the small time asymptotics and state our main result. In Section 4, we prove the exponential equivalence under a stronger condition on the initial value $f$. In Section 5, we relax the condition on the initial value $f$.

\section{Framework}

{\color{rr} We work on a finite-time
interval $[0,1]$ and consider periodic boundary conditions, that is, $x\in \mathbb{T}^d$ where $\mathbb{T}^d=[0,1]^d$ denotes the $d-$dimensional torus.}
Let $\mathcal{L}(K_1,K_2)$ (resp. $\mathcal{L}_2(K_1,K_2)$) be the space of bounded (resp. Hilbert-Schmidt) linear operators from a Hilbert space $K_1$ to another Hilbert space $K_2$, whose norm is denoted by $\|\cdot\|_{\mathcal{L}(K_1, K_2)}$(resp. $\|\cdot\|_{\mathcal{L}_2(K_1, K_2)})$. We will follow closely the framework of \cite{H-Z}. {\color{rr} In the following, we emphasize that the domain of all functional spaces is $\mathbb{T}^d$, for simplicity, we omit it. $C^1$ stands for the space of continuously differentiable functions having bounded first order derivative. $C^1_{lip}$ be the Lipschitz functions in $C^1$.} For $r\in[1,\infty]$, $L^r$ are the usual Lebesgue spaces and $\|\cdot\|_{L^r}$ represents the corresponding norm. In particular, when $p=2$, we write  $H$ for $L^2(\mathbb{T}^d)$. Moreover, we denote by $(\cdot,\cdot)$ and $\|\cdot\|_H$ the inner product and the norm of $L^2(\mathbb{T}^d)$.
In order to measure
higher regularity of functions (in the space variable) we make use of the Bessel potential spaces
$H^{a,r}(\mathbb{T}^d)$, $a\in \mathbb{R}$ and $r\in (1, \infty)$. Throughout the paper we will mostly work with the $L^2-$scale
and so we will write $H^a$ for $H^{a,2}(\mathbb{T}^d)$.
For all $a\geq0$, let
$H^a$ be the usual Sobolev space of order $a$ with the norm
\[
\|v\|^2_{H^a}=\sum_{0\leq|\alpha|\leq a}\int_{\mathbb{T}^d}|D^{\alpha}v|^2dx.
\]
{\color{rr} Here, $\alpha$ is a multi-index, that is, $\alpha=(\alpha_1, \cdot\cdot\cdot, \alpha_d)$
with non-negative integers $\alpha_i, i=1,\cdot\cdot\cdot, d$. $|\alpha|=\sum^d_{i=1}\alpha_i$.}
$H^{-a}$ stands for the topological dual of $H^a$.
Denote by $\langle\cdot,\cdot\rangle$ the duality between $H^1$ and $H^{-1}$.
Moreover, for $u,v\in H^1$, define $((u,v)):=(Du,Dv)$. By Poinc\'{a}re inequality, we know that $((u,u))=\|Du\|^2_H\cong \|u\|^2_{H^1}$.

Now, we introduce the following hypotheses.
\begin{description}
   \item[\textbf{Hypothesis H1}] \quad The flux function $B$, the diffusion matrix $A$, and the noise in (\ref{equ-0}) satisfy:
\begin{description}
  \item[(i)]
$B=(B_1,\cdot\cdot\cdot,B_d): \mathbb{R}\rightarrow \mathbb{R}^d$ is of class $C^1_{lip}$.

  \item[(ii)] $A=(A_{ij})^{d}_{i,j=1}: \mathbb{R}\rightarrow \mathbb{R}^{d\times d}$ is of class $C^1_{lip}$, uniformly positive definite and bounded, i.e. $\varrho I\leq A\leq CI$, with $\varrho>0$.
  \item[(iii)] For each $u\in H, t\in[0,1]$, $\sigma(t,u): U\rightarrow H$ defined by $\sigma(t,u)\bar{e}_k=\sigma_k(t,u(\cdot))$, where $U$ is a separable Hilbert space (with inner product $\langle\cdot,\cdot\rangle_U$ and norm $|\cdot|_U$), $(\bar{e}_k)_{k\geq 1}$ is an orthonormal basis of $U$ and  $\sigma_k(t,\cdot): \mathbb{R}\rightarrow \mathbb{R}$ are real-valued functions.
  In particular, assume that $\sigma$ satisfies the usual linear growth and Lipschitz condition
       \begin{eqnarray}\label{equa-3-1}
      \sum_{k\geq 1}|\sigma_k(t,y)|^2\leq C(1+|y|^2),\quad  \sum_{k\geq1}|\sigma_k(t,y_1)-\sigma_k(t,y_2)|^2\leq C|y_1-y_2|^2 \quad {\rm{for}}\ y, y_1, y_2\in \mathbb{R}.
       \end{eqnarray}
\end{description}
\end{description}
Let $(\Omega, \mathcal{F}, \mathcal{F}_t, P)$ be a stochastic basis with a complete, right-continuous filtration with expectation $E$. The driving process $W(t)$ is a $U-$cylindrical Wiener process defined on this stochastic basis whose paths belong to  $C([0,T],Y)$, where $Y$ is another Hilbert space such that the embedding $U\subset Y$ is Hilbert-Schmidt. $W$ admits the following decomposition
$W(t)=\sum^{\infty}_{k=1}\beta_k(t)\bar{e}_k$, $(\beta_k)_{k\geq 1}$ is a sequence of independent real-valued Brownian motions.
\begin{remark}
The above \textbf{(iii)} implies that $\sigma$ maps $H$ into $\mathcal{L}_2(U,H)$. Indeed, by (\ref{equa-3-1}),
\begin{eqnarray*}
\|\sigma(t,u(\cdot,t))\|^2_{\mathcal{L}_2(U,H)}&=&\sum_{k\geq 1}\|\sigma(t,u(\cdot,t))\bar{e}_k\|^2_H\\
&=& \sum_{k\geq 1}\|\sigma_k(t,u(\cdot,t))\|^2_H=\sum_{k\geq 1}\int_{\mathbb{T}^d}|\sigma_k(t, u(x,t))|^2dx\\
&\leq& C\int_{\mathbb{T}^d}(1+|u(x)|^2)dx\leq C(1+\|u(t)\|^2_H).
\end{eqnarray*}
Thus, for a given predictable process $u\in L^2(\Omega, L^2([0,T],H))$, the stochastic integral $t\rightarrow \int^t_0\sigma(s,u(s))dW(s)$ is a well-defined $H-$valued square integrable martingale.
Moreover, (\ref{equa-3-1}) implies that
\begin{eqnarray*}
\|\sigma(t,u_1(\cdot,t))-\sigma(t,u_2(\cdot,t))\|^2_{\mathcal{L}_2(U,H)}\leq C\|u_1(t)-u_2(t)\|^2_H.
\end{eqnarray*}
Thus, (\ref{equa-3-1}) can be replaced by
\begin{eqnarray}\label{equa-3-1-1}
\|\sigma(t,u(\cdot,t))\|^2_{\mathcal{L}_2(U,H)}\leq C(1+\|u(t)\|^2_H), \quad \|\sigma(t,u_1(\cdot,t))-\sigma(t,u_2(\cdot,t))\|^2_{\mathcal{L}_2(U,H)}\leq C\|u_1(t)-u_2(t)\|^2_H.
\end{eqnarray}
\end{remark}
Now, we recall the definition of a solution to (\ref{equ-0}) from \cite{H-Z}.
\begin{dfn}\label{dfn-1}
An $(\mathcal{F}_t)-$adapted, $H-$valued continuous process $(u(t), t\geq 0)$ is said to be a solution to equation (\ref{equ-0}) if
\begin{description}
  \item[(i)] $u\in L^2(\Omega, C([0,1],H)) \cap L^2(\Omega, L^2( [0,1], H^1))$,
  \item[(ii)] for any $\phi\in C^{\infty}(\mathbb{T}^d)$, $t>0$, the following holds almost surely
  \begin{eqnarray}\notag
  \langle u(t), \phi\rangle&-&\langle f, \phi\rangle-\int^t_0\langle B(u(s)), \nabla \phi\rangle ds\\
\label{equ-2}
  &=&-\int^t_0\langle A(u(s))\nabla u(s), \nabla \phi\rangle ds +\int^t_0\langle \sigma(u(s))dW(s),  \phi\rangle
  \end{eqnarray}
\end{description}
\end{dfn}
With the help of the global well-posedness results in \cite{H-Z}  and  a suitable approximation of initial values by smooth functions in \cite{DHV}, we have
\begin{thm}\label{thm-1-1}
Let the initial value $f\in L^p(\mathbb{T}^d)$ for all $p\in[1,\infty)$. Under the Hypothesis H1, there exists a unique solution to the quasilinear SPDE (\ref{equ-0}) that satisfies the following energy inequality
\begin{eqnarray}\label{equ-3}
E\sup_{0\leq t\leq 1}\|u(t)\|^2_H+\int^1_0 E\|u(t)\|^2_{H^1}dt<\infty.
\end{eqnarray}
\end{thm}
\section{A priori estimates}
In order to prove the small asymptotics of the solution to (\ref{equ-0}), we need to make some higher Sobolev norm estimates of $u$.
\begin{lemma}\label{lem-1}
For any initial value $f\in  L^p(\mathbb{T}^d)$ for all $p\in[1,\infty)$. Suppose the Hypothesis H1 holds, then for any $p\in[1,\infty)$,
\begin{eqnarray}\label{eqqq-67}
E\sup_{0\leq t\leq 1}\|u(t)\|^{2p}_H+E\int^1_0\|u(t)\|^{2(p-1)}_H \|u(t)\|^2_{H^1}dt<\infty,
\end{eqnarray}
and
\begin{eqnarray}\label{eqqq-58}
E\Big(\int^1_0\|u(t)\|^2_{H^1}dt\Big)^p<\infty.
\end{eqnarray}
\end{lemma}
\begin{proof}
Applying It\^{o} formula to $\|u(t)\|^{2}_H$, we get
\begin{eqnarray*}
\|u(t)\|^2_H&=&\|f\|^2_H-2 \int^t_0\langle u(s), div (B(u(s)))\rangle ds +2\int^t_0\langle u(s), div \big(A(u(s))\nabla u(s)\big)\rangle ds\\
&&\ +2\int^t_0 \langle u(s), \sigma(s, u(s))dW(s)\rangle + \int^t_0 \|\sigma( s, u(s))\|^2_{\mathcal{L}_2(U,H)}ds\\
&=& \|f\|^2_H+2 \int^t_0\langle \nabla u(s), B(u(s))\rangle ds-2\int^t_0\langle \nabla u(s), A(u(s))\nabla u(s)\rangle ds\\
&&\ +2\int^t_0 \langle u(s), \sigma(s, u(s))dW(s)\rangle+ \int^t_0 \|\sigma( s, u(s))\|^2_{\mathcal{L}_2(U,H)}ds.
\end{eqnarray*}
Then, it yields that
\begin{eqnarray*}
\langle\|u\|^{2}_H\rangle_s=4\int^s_0\|u(l)\|^{2}_H\|\sigma(l,u(l))\|^{2}_{\mathcal{L}_2(U,H)}dl.
\end{eqnarray*}
Employing  It\^{o} formula again to $\|u(t)\|^{2p}_H$, it gives that
\begin{eqnarray*}
\|u(t)\|^{2p}_H&=&\|f\|^{2p}_H+p\int^t_0\|u(s)\|^{2(p-1)}_Hd\|u(s)\|^{2}_H+\frac{1}{2}p(p-1)\int^t_0\|u(s)\|^{2(p-2)}_Hd\langle\|u\|^{2}_H\rangle_s\\
&=&\|f\|^{2p}_H+2p\int^t_0\|u(s)\|^{2(p-1)}_H\langle \nabla u(s), B(u(s))\rangle ds\\
&&\ -2p\int^t_0\|u(s)\|^{2(p-1)}_H\langle \nabla u(s), A(u(s))\nabla u(s)\rangle ds\\
&&\ +2p\int^t_0\|u(s)\|^{2(p-1)}_H\langle u(s), \sigma(s, u(s))dW(s)\rangle\\
&&\ +p\int^t_0\|u(s)\|^{2(p-1)}_H\|\sigma( s, u(s))\|^2_{\mathcal{L}_2(U,H)}ds\\
&&\ +2p(p-1)\int^t_0\|u(s)\|^{2(p-2)}_H\|u(s)\|^{2}_H\|\sigma(s,u(s))\|^{2}_{\mathcal{L}_2(U,H)}ds\\
&:=& \|f\|^{2p}_H+I_1(t)+I_2(t)+I_3(t)+I_4(t)+I_5(t).
\end{eqnarray*}
By Hypothesis H1 and Young's inequality, we get
\begin{eqnarray*}
I_1(t)&\leq& 2pC\int^t_0\|u(s)\|^{2(p-1)}_H\|u(s)\|_{H^1}(1+\|u(s)\|_H) ds\\
&\leq& p\varrho \int^t_0\|u(s)\|^{2(p-1)}_H\|u(s)\|^2_{H^1}ds+\frac{pC}{\varrho}t+\frac{pC}{\varrho}\int^t_0\|u(s)\|^{2p}_Hds.
\end{eqnarray*}
Utilizing Hypothesis H1, it follows that
\begin{eqnarray*}
I_2(t)\leq -2p\varrho \int^t_0\|u(s)\|^{2(p-1)}_H\|u(s)\|^2_{H^1}ds.
\end{eqnarray*}
Moreover, by (\ref{equa-3-1-1}), it yields that
\begin{eqnarray*}
I_4(t)\leq pCt+pC\int^t_0\|u(s)\|^{2p}_Hds,
\end{eqnarray*}
and
\begin{eqnarray*}
I_5(t)&\leq& 2p(p-1)C\int^t_0\|u(s)\|^{2(p-2)}_H\|u(s)\|^{2}_H(1+\|u(s)\|^{2}_H)ds\\
&\leq& 2p(p-1)Ct+2p(p-1)C\int^t_0\|u(s)\|^{2p}_Hds.
\end{eqnarray*}

Collecting all the above estimates, we get
\begin{eqnarray*}
&&\|u(t)\|^{2p}_H+p\varrho \int^t_0\|u(s)\|^{2(p-1)}_H\|u(s)\|^2_{H^1}ds\\
&\leq& \|f\|^{2p}_H+\frac{pC}{\varrho}t+pCt+ 2p(p-1)Ct+\Big(\frac{pC}{\varrho}+pC+ 2p(p-1)C\Big)\int^t_0\|u(s)\|^{2p}_Hds+|I_3(t)|.
\end{eqnarray*}
Hence, we conclude that
\begin{eqnarray}\notag
&&E\left[\sup_{0\leq s\leq 1}\|u(s)\|^{2p}_H+p\varrho \int^1_0\|u(s)\|^{2(p-1)}_H\|u(s)\|^2_{H^1}ds\right]\\ \notag
&\leq& \|f\|^{2p}_H+\frac{pC}{\varrho}+pC+ 2p(p-1)C\\
\label{eqqq-68}
&&\ +\Big(\frac{pC}{\varrho}+pC+ 2p(p-1)C\Big)\int^1_0E\sup_{0\leq s\leq t}\|u(s)\|^{2p}_Hdt+E\sup_{0\leq s\leq 1}|I_3(s)|.
\end{eqnarray}
Applying the Burkholder-Davis-Gundy inequality, we deduce that
\begin{eqnarray}\notag
E\sup_{0\leq s\leq 1}|I_3(s)|&\leq&2pC_pE\int^1_0\|u(s)\|^{4(p-1)}_H\|u(s)\|^2_H(1+\|u(s)\|^2_H)ds\\ \notag
&\leq& 2pC_pE\Big[\sup_{0\leq s\leq 1}\|u(s)\|^{p}_H\Big(\int^1_0\|u(s)\|^{2(p-1)}_H(1+\|u(s)\|^2_H)ds\Big)^{\frac{1}{2}}\Big]\\ \notag
&\leq& 2pC_pE\Big[\sup_{0\leq s\leq 1}\|u(s)\|^{p}_H\Big(1+\int^1_0\|u(s)\|^{2p}_Hds\Big)^{\frac{1}{2}}\Big]\\
\label{eqqq-69}
&\leq& \frac{1}{2}E\sup_{0\leq s\leq 1}\|u(s)\|^{2p}_H+2pC+2pCE\int^1_0\|u(s)\|^{2p}_Hds.
\end{eqnarray}
Putting (\ref{eqqq-69}) into  (\ref{eqqq-68}), it yields that
\begin{eqnarray*}\notag
&&E\Big[\sup_{0\leq s\leq 1}\|u(s)\|^{2p}_H+2p\varrho \int^1_0\|u(s)\|^{2(p-1)}_H\|u(s)\|^2_{H^1}ds\Big]\\ \notag
&\leq& 2\|f\|^{2p}_H+\frac{2pC}{\varrho}+2pC+ 4p(p-1)C\\
&&\ +\Big(\frac{2pC}{\varrho}+2pC+ 4p(p-1)C\Big)\int^1_0E\Big[\sup_{0\leq l\leq s}\|u(l)\|^{2p}_H+2p\varrho \int^s_0\|u(l)\|^{2(p-1)}_H\|u(l)\|^2_{H^1}dl\Big]ds.
\end{eqnarray*}
Applying Gronwall inequality, we get
\begin{eqnarray*}
&&E\left[\sup_{0\leq s\leq 1}\|u(s)\|^{2p}_H+2p\varrho \int^1_0\|u(s)\|^{2(p-1)}_H\|u(s)\|^2_{H^1}ds\right]\\ \notag
&\leq& \Big[2\|f\|^{2p}_H+\frac{2pC}{\varrho}+2pC+ 4p(p-1)C\Big]\exp\Big\{\frac{2pC}{\varrho}+2pC+ 4p(p-1)C\Big\},
\end{eqnarray*}
which implies (\ref{eqqq-67}) is valid.

For the special case $p=2$, we have
\begin{eqnarray*}
&&\|u(t)\|^2_H+\varrho\int^t_0\|u(s)\|^2_{H^1}ds\\
&\leq& \|f\|^2_H+\frac{C}{\varrho}t+Ct+\Big(\frac{C}{\varrho}+C\Big) \int^t_0 \|u(s)\|^2_{H}ds
+2\Big|\int^t_0 \langle u(s), \sigma(s, u(s))dW(s)\rangle\Big|.
\end{eqnarray*}
Then, it gives that
\begin{eqnarray}\notag
&&\varrho^p E\left(\int^1_0\|u(s)\|^2_{H^1}ds\right)^{p}\\ \notag
&\leq& C_p\Big(\|f\|^2_H+\frac{C}{\varrho}+C\Big)^{p}+C_p\Big(\frac{C}{\varrho}+C\Big)^{p}E\left(\int^1_0 \|u(s)\|^2_{H}ds\right)^p\\ \notag
&&\
+2C_p  E\sup_{0\leq s\leq 1}\Big|\int^s_0 \langle u(l), \sigma(l, u(l))dW(l)\rangle\Big|^{p}\\ \notag
&\leq& C_p\Big(\|f\|^2_H+\frac{C}{\varrho}+C\Big)^{p}+C_p\Big(\frac{C}{\varrho}+C\Big)^{p}E\sup_{0\leq s\leq 1} \|u(s)\|^{2p}_{H}\\
\label{eqqq-71}
&&\
+2C_p  E\sup_{0\leq s\leq 1}\Big|\int^s_0 \langle u(l), \sigma(l, u(l))dW(l)\rangle\Big|^{p}.
\end{eqnarray}
By the Burkholder-Davis-Gundy inequality, we reach
\begin{eqnarray}\notag
&&2C_p  E\sup_{0\leq s\leq 1}\Big|\int^s_0 \langle u(l), \sigma(l, u(l))dW(l)\rangle\Big|^{p}\\ \notag
&\leq&2C_pE\Big(\int^1_0\big(1+\|u(s)\|^2_H\big)ds\Big)^{\frac{p}{2}}\\
\label{eqqq-72}
&\leq& 2C_p+2C_pE\sup_{0\leq s\leq 1}\|u(s)\|^{2p}_H.
\end{eqnarray}
Putting (\ref{eqqq-72}) into (\ref{eqqq-71}), we obtain
\begin{eqnarray}\notag
&&\varrho^p E\left(\int^1_0\|u(s)\|^2_{H^1}ds\right)^{p}\\ \notag
&\leq& C_p\Big(\|f\|^2_H+\frac{C}{\varrho}+C\Big)^{p}+2C_p+\Big[2C_p+C_p\Big(\frac{C}{\varrho}+C\Big)^{p}\Big]E\sup_{0\leq s\leq 1} \|u(s)\|^{2p}_{H}.
\end{eqnarray}
By (\ref{eqqq-67}), we get
\begin{eqnarray*}
E\left(\int^1_0\|u(s)\|^2_{H^1}ds\right)^{p}\leq C(p,\varrho),
\end{eqnarray*}
hence, we complete the proof of (\ref{eqqq-58}).
\end{proof}

\section{Small time asymptotics and the statement of our main result}
Let $\varepsilon>0$, by the scaling property of the Brownian motion, it is easy to see that $u(\varepsilon t)$ coincides in law with the solution of the following equation:
\begin{eqnarray}\notag
u^{\varepsilon}_f(t)+\varepsilon\int^t_0 div (B(u^{\varepsilon}_f(s)))ds&=&f(x)+\varepsilon \int^t_0 div\Big(A(u^{\varepsilon}_f(s))\nabla u^{\varepsilon}_f(s)\Big)ds\\
\label{eqq-5}
&&\ +\sqrt{\varepsilon} \int^t_0 \sigma(\varepsilon s, u^{\varepsilon}_f(s))dW(s).
\end{eqnarray}
Let $\mu^{\varepsilon}_f$ be the law of $u^{\varepsilon}_f(\cdot)$ on $C([0,1],H)$. Define a functional $I(g)$ on $C([0,1],H)$ by
\begin{eqnarray}\label{eqq-6}
I(g)=\inf_{h\in\Gamma_g}\Big\{\frac{1}{2}\int^1_0|\dot{h}(t)|^2_{U}dt\Big\},
\end{eqnarray}
where
\begin{eqnarray*}
\Gamma_g&=&\Big\{h\in C([0,1],U): h(\cdot) \ is\ absolutely\ continuous\ and\ such\ that\ \\
&& g(t)=f+\int^t_0 \sigma(s, g(s))\dot{h}(s)ds, 0\leq t\leq 1 \Big \}.
\end{eqnarray*}

To establish the small time asymptotics of the quasilinear SPDE (\ref{equ-0}), we need some additional conditions on the diffusion coefficient $\sigma$.
\begin{description}
  \item[\textbf{Hypothesis H2}] For each $u\in H^1$ and $t\in [0,1]$, the mapping $\sigma(t,u):U\rightarrow H^1$ satisfies
  \begin{eqnarray}\label{e-1}
  \|\sigma(t,u)\|^2_{\mathcal{L}_2(U,H^1)}\leq C(1+\|u\|^2_{H^1}),
  \end{eqnarray}
  and for $u_1, u_2\in H^1$,
   \begin{eqnarray}\label{e-2}
  \|\sigma(t,u_1)-\sigma(t,u_2)\|^2_{\mathcal{L}_2(U,H^1)}\leq C\|u_1-u_2\|^2_{H^1}.
  \end{eqnarray}
\end{description}

Now, we recall some notations and a condition on the diffusion coefficient $\sigma$ from \cite{H-Z}.
Denote by $\zeta(K,X)$ the space of the $\zeta-$radonifying operators from a separable Hilbert space $K$ to a $2-$smooth Banach space
$X$ ( with the norm  $\|\cdot\|_X$). We recall that $\Psi\in \zeta(K,X)$ if the series
\[
\sum_{k\geq 0}\zeta_k \Psi(e_k)
\]
converges in $L^2(\tilde{\Omega}, X)$, for any sequence $(\zeta_k)_{k\geq0}$ of independent Gaussian real-valued random
variables on a probability space $(\tilde{\Omega},\tilde{\mathcal{F}},\tilde{P} )$ and any orthonormal basis $(e_k)_{k\geq 0}$ of $K$. Then, this
space is endowed with the norm
\begin{eqnarray*}
\|\Psi\|_{\zeta(K,X)}:=\Big(\tilde{E}\|\sum^{\infty}_{k=1}\zeta_k\Psi(e_k)\|^2_X\Big)^{\frac{1}{2}}, \quad \Psi\in \zeta(K,X).
\end{eqnarray*}
The space $\zeta(K,X)$ does not depend on $(\zeta_k)_{k\geq 1}$, nor on $(\zeta_k)_{k\geq 1}$ and is a Banach space.

Recall that the Bessel potential spaces $H^{a,b}$ with $a\geq 0$ and $b\in [2, \infty)$
belong to the class of $2-$smooth Banach spaces and hence they are well suited for the stochastic
It\^{o} integration (see \cite{Br}, \cite{BP} for the precise construction of the stochastic integral).
With this notation in hand, we state our last assumption upon the coefficient $\sigma$ to ensure the existence of the stochastic integral in (\ref{equ-0}) as an
$H^{a,b}$-valued process.
\begin{description}
   \item[\textbf{Hypothesis H3}] \quad For $a<2$, $b\in [2,\infty)$ and any $t\in [0,1]$,
   \begin{eqnarray}\label{eqq-1}
   \|\sigma(t,u)\|_{\zeta(U, H^{a,b})}\leq \left\{
                                            \begin{array}{ll}
                                              C(1+\|u\|_{H^{a,b}}), & a\in [0,1], \\
                                              C(1+\|u\|_{H^{a,b}}+\|u\|^a_{H^{1,ab}}), & a>1.
                                            \end{array}
                                          \right.
\end{eqnarray}
\end{description}
Detailed discussion of this condition can be found in \cite{DDMH,H-Z}.

Recall that $\mu^{\varepsilon}_f$ is the law of $u^{\varepsilon}_f(\cdot)$ on $C([0,1],H)$. The main result of this article reads as follows.
\begin{thm}\label{thm-1}
 For any initial value $f\in L^p(\mathbb{T}^d)$ for all $p\in [1,\infty)$, under Hypotheses H1-H3, $\mu^{\varepsilon}_f$ satisfies a large deviation principle with the rate function $I(\cdot)$ defined by (\ref{eqq-6}), that is,
\begin{description}
  \item[(i)] For any closed subset $F\subset C([0,1],H) $,
  \[
  \limsup_{\varepsilon\rightarrow 0}\varepsilon \log \mu^{\varepsilon}_{f}(F)\leq -\inf_{g\in F}I(g).
  \]
  \item[(ii)] For any open subset $G\subset C([0,1],H) $,
  \[
  \liminf_{\varepsilon\rightarrow 0}\varepsilon \log \mu^{\varepsilon}_{f}(G)\geq -\inf_{g\in G}I(g).
  \]
\end{description}
\end{thm}
\begin{proof}
 Let $v^{\varepsilon}_f$ be the solution of the stochastic equation
\begin{eqnarray}\label{eq-7}
v^{\varepsilon}_f(t)=f(x)+\sqrt{\varepsilon}\int^t_0\sigma(\varepsilon s, v^{\varepsilon}_f(s))dW(s),
\end{eqnarray}
and $\vartheta^{\varepsilon}_f$ be the law of $v^{\varepsilon}_f(\cdot)$ on $C([0,1],H)$. By \cite{Daprato}, we know that $\vartheta^{\varepsilon}_f$ satisfies a large deviation principle with the rate function $I(\cdot)$. Based on Theorem 4.2.13 in \cite{DZ}, it suffices to show that two families of the probability measures $\mu^{\varepsilon}_f$ and $\vartheta^{\varepsilon}_f$ are exponentially equivalent, that is, for any $\delta >0$,
\begin{eqnarray}\label{eq-8}
\lim_{\varepsilon\rightarrow 0}\varepsilon \log P\Big(\sup_{0\leq t\leq 1}\|u^{\varepsilon}_f(t)-v^{\varepsilon}_f(t)\|^2_H>\delta\Big)=-\infty.
\end{eqnarray}
In the following, we will divide the proof of (\ref{eq-8}) into two steps.  Firstly, we prove (\ref{eq-8}) holds under a stronger condition on the initial value $f$ in Section \ref{section-1} (see Proposition \ref{prp-0}). Secondly, we succeed to verify (\ref{eq-8}) holds without the stronger condition in Section \ref{section-2} (see Proposition \ref{prp-2}).

\end{proof}

\textbf{From now on, for the sake of simplicity, we denote that $u^{\varepsilon}=u^{\varepsilon}_f$ and $v^{\varepsilon}=v^{\varepsilon}_f$ when the initial value is not emphasized.}

\subsection{A priori estimates}
In order to prove  (\ref{eq-8}), a priori estimates of $u^{\varepsilon}$ and $v^{\varepsilon}$ are necessary.

Define
\[
\Big(|u^{\varepsilon}|^H_{H^1}(1)\Big)^2=\sup_{0\leq t\leq  1}\|u^{\varepsilon}(t)\|^2_H+\varepsilon \varrho \int^1_0 \|u^{\varepsilon}(t)\|^2_{H^1}dt.
\]
The following result is an estimation of the probability that the solution of (\ref{eqq-5}) leaves an
energy ball.
\begin{lemma}\label{lemm-1}
 For any initial value $f\in L^p(\mathbb{T}^d)$ for all $p\in [1,\infty)$, under Hypothesis H1, it holds true that
\begin{eqnarray}\label{eqq-7}
\lim_{M\rightarrow \infty}\sup_{0<\varepsilon \leq 1}\varepsilon \log P\Big((|u^{\varepsilon}|^H_{H^1}(1))^2>M\Big)=-\infty.
\end{eqnarray}
\end{lemma}

\begin{proof}
Applying It\^{o} formula, we get
\begin{eqnarray*}
\|u^{\varepsilon}(t)\|^2_H&=&\|f\|^2_H-2\varepsilon \int^t_0\langle u^{\varepsilon}(s), div (B(u^{\varepsilon}(s)))\rangle ds\\
&&\ +2\varepsilon \int^t_0\langle u^{\varepsilon}(s), div (A(u^{\varepsilon}(s))\nabla u^{\varepsilon}(s))\rangle ds+2\sqrt{\varepsilon}\int^t_0 \langle u^{\varepsilon}(s), \sigma(\varepsilon s, u^{\varepsilon}(s))dW(s)\rangle\\
&&\ +\varepsilon \int^t_0 \|\sigma(\varepsilon s, u^{\varepsilon}(s))\|^2_{\mathcal{L}_2(U,H)}ds\\
&:=& \|f\|^2_H+I_1(t)+I_2(t)+I_3(t)+I_4(t).
\end{eqnarray*}
By integration by parts, the Lipschitz property of $B$, the Cauchy-Schwarz inequality and Young's inequality, we get
\begin{eqnarray*}
I_1(t)&=&2\varepsilon \int^t_0\langle \nabla u^{\varepsilon}(s), B(u^{\varepsilon}(s))\rangle ds\\
&\leq & 2\varepsilon C\int^t_0 \|u^{\varepsilon}(s)\|_{H^1}(1+\|u^{\varepsilon}(s)\|_{H})ds\\
&\leq & \varepsilon \varrho \int^t_0 \|u^{\varepsilon}(s)\|^2_{H^1}ds+\frac{\varepsilon C}{\varrho}\int^t_0 \big(1+\|u^{\varepsilon}(s)\|^2_{H}\big)ds,
\end{eqnarray*}
where $\varrho$ is the positive constant appeared in Hypothesis H1.

By integration by parts and (ii) of Hypothesis H1,
\begin{eqnarray*}
I_2(t)&=&-2\varepsilon \int^t_0\langle \nabla u^{\varepsilon}(s), A(u^{\varepsilon}(s))\nabla u^{\varepsilon}(s) \rangle ds\\
&\leq & -2\varepsilon \varrho \int^t_0 \|u^{\varepsilon}(s)\|^2_{H^1}ds.
\end{eqnarray*}
With the help of (\ref{equa-3-1-1}), it follows that
\begin{eqnarray*}
I_4(t)\leq \varepsilon C\int^t_0(1+ \|u^{\varepsilon}(s)\|^2_{H})ds.
\end{eqnarray*}
Collecting all the above estimates, we deduce that
\begin{eqnarray*}
&&\|u^{\varepsilon}(t)\|^2_H+\varepsilon \varrho \int^t_0 \|u^{\varepsilon}(s)\|^2_{H^1}ds\\
&\leq& \|f\|^2_H+\varepsilon C t+\frac{\varepsilon C}{\varrho}t+\Big(\varepsilon C+\frac{\varepsilon C}{\varrho}\Big)\int^t_0 \|u^{\varepsilon}(s)\|^2_{H}ds+|I_3(t)|.
\end{eqnarray*}
Then, it gives that
\begin{eqnarray*}
(|u^{\varepsilon}|^H_{H^1}(1))^2
\leq\|f\|^2_H+\varepsilon C +\frac{\varepsilon C}{\varrho}+\Big(\varepsilon C+\frac{\varepsilon C}{\varrho}\Big)\int^1_0 (|u^{\varepsilon}|^H_{H^1}(s))^2ds+\sup_{t\in [0,1]}|I_3(t)|.
\end{eqnarray*}
Hence, for $p\geq 2$, we have
\begin{eqnarray}\notag
\Big(E(|u^{\varepsilon}|^H_{H^1}(1))^{2p}\Big)^{\frac{1}{p}}
&\leq& \Big(\|f\|^2_H+\varepsilon C +\frac{\varepsilon C}{\varrho}\Big)+\Big(\varepsilon C+\frac{\varepsilon C}{\varrho}\Big)\Big(E\big(\int^1_0 (|u^{\varepsilon}|^H_{H^1}(s))^2ds\big)^p\Big)^{\frac{1}{p}}\\
\label{eqqq-1}
&&\ +\Big(E\big(\sup_{t\in [0,1]}|I_3(t)|^p\big)\Big)^{\frac{1}{p}}.
\end{eqnarray}
To estimate the stochastic integral term in (\ref{eqqq-1}), we will use the following remarkable result from \cite{B-Y} and \cite{Davis} that there exists a universal constant $C$ such that, for any $p\geq 2$ and for any continuous martingale $M_t$ with $M_0=0$, one has
\begin{eqnarray}\label{eqqq-2}
\Big(E(|M^*_t|^p)\Big)^{\frac{1}{p}}\leq Cp^{\frac{1}{2}}\Big(E\langle M\rangle^{\frac{p}{2}}_t\Big)^{\frac{1}{p}},
\end{eqnarray}
where $M^*_t=\sup_{s\in [0,t]}|M_s|$.

Using (\ref{eqqq-2}), we get
\begin{eqnarray}\notag
&&\Big(E(\sup_{t\in [0,1]}|I_3(t)|^p)\Big)^{\frac{1}{p}}\\ \notag
&=& 2\sqrt{\varepsilon}\Big(E\Big(\sup_{0\leq t\leq 1}\int^t_0 \langle u^{\varepsilon}(s), \sigma(\varepsilon s, u^{\varepsilon}(s))dW(s)\rangle\Big)^p\Big)^{\frac{1}{p}}\\ \notag
&\leq& 2C\sqrt{p \varepsilon}\Big(E\Big(\int^1_0 \|u^{\varepsilon}(s)\|^2_H\|\sigma(\varepsilon s, u^{\varepsilon}(s))\|^2_{\mathcal{L}_2(U,H)}ds\Big)^{\frac{p}{2}}\Big)^{\frac{1}{p}}\\ \notag
&\leq& 2C\sqrt{p \varepsilon}\Big(E\Big(\int^1_0 \|u^{\varepsilon}(s)\|^2_H(1+\|u^{\varepsilon}(s)\|^2_H)ds\Big)^{\frac{p}{2}}\Big)^{\frac{1}{p}}\\
\notag
&\leq& 2C\sqrt{p \varepsilon}\Big[\Big(E\Big(\int^1_0 \big(1+\|u^{\varepsilon}(s)\|^2_H\big)^2ds\Big)^{\frac{p}{2}}\Big)^{\frac{2}{p}}\Big]^{\frac{1}{2}}\\
\notag
&\leq& 2C\sqrt{p \varepsilon}\Big[\Big(E(\int^1_0 (1+\|u^{\varepsilon}(s)\|^4_H)ds)^{\frac{p}{2}}\Big)^{\frac{2}{p}}\Big]^{\frac{1}{2}}\\
\label{eqqq-3}
&\leq& 2C\sqrt{p \varepsilon}\Big[\int^1_0 1+(E\|u^{\varepsilon}(s)\|^{2p}_H)^{\frac{2}{p}}ds\Big]^{\frac{1}{2}},
\end{eqnarray}
where $C$ is a constant which may change from line to line. On the other hand,
\begin{eqnarray}\label{eqqq-4}
2\varepsilon C \left(E\Big(\int^1_0\Big(|u^{\varepsilon}|^H_{H^1}(s)\Big)^2ds\Big)^p\right)^{\frac{1}{p}}\leq 2\varepsilon C \int^1_0 \Big(E\Big(|u^{\varepsilon}|^H_{H^1}(s)\Big)^{2p}\Big)^{\frac{1}{p}}ds.
\end{eqnarray}
Combining (\ref{eqqq-1}), (\ref{eqqq-3}) and (\ref{eqqq-4}), we arrive at
\begin{eqnarray}\notag
\Big(E\Big(|u^{\varepsilon}|^H_{H^1}(1)\Big)^{2p}\Big)^{\frac{2}{p}}
&\leq& 4\Big(\|f\|^2_H+\varepsilon C +\frac{\varepsilon C}{\varrho}\Big)^2+4\Big(\varepsilon C+\frac{\varepsilon C}{\varrho}\Big)\int^1_0\Big(E \Big(|u^{\varepsilon}|^H_{H^1}(s)\Big)^{2p}\Big)^{\frac{2}{p}}ds\\
\label{eqqq-5}
&&\ +16 C^2 \varepsilon p+16 C^2 \varepsilon p\int^1_0\Big(E \Big(|u^{\varepsilon}|^H_{H^1}(s)\Big)^{2p}\Big)^{\frac{2}{p}}ds.
\end{eqnarray}
Applying Gronwall inequality to (\ref{eqqq-5}), it yields
\begin{eqnarray}\label{eqqq-6}
\Big(E(|u^{\varepsilon}|^H_{H^1}(1))^{2p}\Big)^{\frac{2}{p}}
\leq \Big[4\Big(\|f\|^2_H+\varepsilon C +\frac{\varepsilon C}{\varrho}\Big)^2+16 C^2 \varepsilon p\Big]\cdot \exp\left\{4\Big(\varepsilon C+\frac{\varepsilon C}{\varrho}\Big)+16 C^2 \varepsilon p\right\}.
\end{eqnarray}
Taking $p=\frac{1}{\varepsilon}$ in (\ref{eqqq-6}) and using Chebyshev inequality, it follows that
\begin{eqnarray*}
&&\varepsilon \log P\Big((|u^{\varepsilon}|^H_{H^1}(1))^{2}>M\Big)\\
&\leq& -\log M+\log\Big(E\big(|u^{\varepsilon}|^H_{H^1}(1)\big)^{2p}\Big)^{\frac{1}{p}}\\
&\leq&  -\log M+ \log\sqrt{\Big[4\Big(\|f\|^2_H+\varepsilon C +\frac{\varepsilon C}{\varrho}\Big)^2+16 C^2 \varepsilon p\Big]}+2(\varepsilon C+\frac{\varepsilon C}{\varrho})+8 C^2 \varepsilon p.
\end{eqnarray*}
Thus,
\begin{eqnarray*}
&&\sup_{0<\varepsilon \leq 1}\varepsilon \log P\Big((|u^{\varepsilon}|^H_{H^1}(1))^{2}>M\Big)\\
&\leq&  -\log M+ \log\sqrt{\Big[4\Big(\|f\|^2_H+ C +\frac{ C}{\varrho}\Big)^2+16 C^2p\Big]}+2\Big( C+\frac{ C}{\varrho}\Big)+8 C^2 p.
\end{eqnarray*}
Letting $M\rightarrow \infty$ on both side of the above inequality, we obtain the desired result.

\end{proof}
Since $H^1$ is dense in $H$, there exits a sequence $\{f_n\}^{\infty}_{n=1}\subset H^1$ such that
\[
\lim_{n\rightarrow \infty}\|f_n-f\|_H=0.
\]
Let $u^{\varepsilon}_n$ be the solution of (\ref{eqq-5}) with the initial value $f_n$. From the proof process of Lemma \ref{lemm-1}, it is easily to deduce that
\begin{eqnarray}\label{e-4}
\lim_{M\rightarrow \infty}\sup_n\sup_{0<\varepsilon\leq 1}\varepsilon \log P\Big((|u^{\varepsilon}_n|^H_{H^1}(1))^2>M\Big)=-\infty.
\end{eqnarray}
Let $v^{\varepsilon}_n$ be the solution of (\ref{eq-7}) with the initial value $f_n$. We have the following result.

\begin{lemma}\label{lemm-7}
Under Hypotheses H1 and H2, for any $n\in \mathbb{Z}^+$, we have
\begin{eqnarray}\label{e-3}
\lim_{M\rightarrow \infty}\sup_{0<\varepsilon \leq 1}\varepsilon \log P\Big(\sup_{0\leq t\leq 1}\|v^{\varepsilon}_n(t)\|^2_{H^1}>M\Big)=-\infty.
\end{eqnarray}
\end{lemma}
\begin{proof}
Applying It\^{o} formula to $\|v^{\varepsilon}_n(t)\|^2_{H^1}$, one obtains
\begin{eqnarray*}
\|v^{\varepsilon}_n(t)\|^2_{H^1}
=\|f_n\|^2_{H^1}+2\sqrt{\varepsilon}\int^t_0\Big(\Big(v^{\varepsilon}_n(s), \sigma(\varepsilon s, v^{\varepsilon}_n(s))dW(s)\Big)\Big)
+\varepsilon \int^t_0\|\sigma(\varepsilon s, v^{\varepsilon}_n(s))\|^2_{\mathcal{L}_2(U,H^1)}ds.
\end{eqnarray*}
By Hypothesis H2 and (\ref{eqqq-2}), we deduce that
\begin{eqnarray*}
\Big(E\Big[\sup_{0\leq t\leq r}\|v^{\varepsilon}_n(t)\|^{2p}_{H^1}\Big]\Big)^{\frac{2}{p}}
&\leq& 2\|f_n\|^4_{H^1}+8C\varepsilon p\Big(E\Big(\int^r_0\|v^{\varepsilon}_n(t)\|^{2}_{H^1}\|\sigma(\varepsilon s, v^{\varepsilon}_n(s))\|^2_{\mathcal{L}_2(U,H^1)}ds\Big)^{\frac{p}{2}}\Big)^{\frac{2}{p}}\\
&&\ +4\varepsilon^2 Cr \Big(r+\int^r_0 \Big(E\Big[\sup_{0\leq l\leq s}\|v^{\varepsilon}_n(l)\|^{2p}_{H^1}\Big]\Big)^{\frac{2}{p}}ds\Big)\\
&\leq& 2\|f_n\|^4_{H^1}+16C\varepsilon p\Big(r+\int^r_0\Big(E\Big[\sup_{0\leq l\leq s}\|v^{\varepsilon}_n(l)\|^{2p}_{H^1}\Big]\Big)^{\frac{2}{p}}ds\Big)\\
&&\ +4\varepsilon^2 Cr \Big(r+\int^r_0 \Big(E[\sup_{0\leq l\leq s}\|v^{\varepsilon}_n(l)\|^{2p}_{H^1}]\Big)^{\frac{2}{p}}ds\Big).
\end{eqnarray*}
By Gronwall inequality, we get
\begin{eqnarray*}
\Big(E\Big[\sup_{0\leq t\leq 1}\|v^{\varepsilon}_n(t)\|^{2p}_{H^1}\Big]\Big)^{\frac{2}{p}}\leq
\Big(2\|f_n\|^4_{H^1}+16C\varepsilon p+4\varepsilon^2 C^2\Big)e^{16C\varepsilon p+4\varepsilon^2 C^2}.
\end{eqnarray*}
The rest of the proof is the same as Lemma \ref{lemm-1}, we omit it.
\end{proof}
\begin{remark}
To obtain the estimation of $\|v^{\varepsilon}_n(t)\|^2_{H^1}$, the Hypothesis H2 is necessary, since there is no Stokes operator in (\ref{eq-7}).
\end{remark}
\section{Proof of (\ref{eq-8}) under a stronger condition  }\label{section-1}
As stated in Theorem \ref{thm-1}, we firstly  prove  (\ref{eq-8}) holds under a stronger conditions on the initial value $f$. The proof of (\ref{eq-8}) is quite involved, because the coefficients of (\ref{eqq-5}) are neither monotone nor locally monotone. Inspired by \cite{H-Z}, we introduce the heat kernel to smooth the operator $A$.

Let $P_r, r>0$ denote the semigroup on $H$ generated by the Laplacian on $\mathbb{T}^d$. Recall that
\[
P_r g(x)=\int_{\mathbb{T}^d}P_r(x,z)g(z)dz,
\]
where $P_r(x,z)$ stands for the heat kernel, $x, z\in\mathbb{T}^d$.

Referring to (4.3) in \cite{H-Z}, the following property
\begin{eqnarray}\label{eqqq-13}
\|P_r g\|_{L^{\infty}}\leq C_r \|g\|_H,\quad g\in H,
\end{eqnarray}
is valid.
For $r>0$, $u\in H$, set
\begin{eqnarray}\label{equ-06}
A_r(u)(x)=P_r(A(u))(x), \ x\in\mathbb{T}^d,
\end{eqnarray}
where
\[
P_r(A(u))(x)=(P_r(A_{ij}(u))(x))^d_{i,j=1}.
\]
Note that there exists a constant $C$, independent of $r$,  such that
\begin{eqnarray}\label{equ-07}
\varrho |\xi|^2\leq A_r(u)(x)\xi\cdot\xi\leq C|\xi|^2\quad \forall r>0,\  u\in H,\ x\in\mathbb{T}^d, \ \xi\in\mathbb{R}^d.
\end{eqnarray}

Employing the operator $P_r$, the equation (\ref{eqq-5}) is changed to be
 \begin{eqnarray}\label{eqq-2}
\left\{
  \begin{array}{ll}
   du^{r,\varepsilon}+\varepsilon div(B(u^{r,\varepsilon}))dt=\varepsilon div(A_r(u^{r,\varepsilon})\nabla u^{r,\varepsilon})dt+\sqrt{\varepsilon}\sigma(\varepsilon t,u^{r,\varepsilon}(t))dW(t),
  & \ x\in \mathbb{T}^d,\  t\in [0,1] \\
   u^{r,\varepsilon}(0)=f. &
  \end{array}
\right.
\end{eqnarray}
Applying similar arguments as Lemma \ref{lemm-1}, we deduce that
\begin{lemma}\label{lemm-2}
  For any initial value $f\in L^p(\mathbb{T}^d)$ for all $p\in [1,\infty)$, under Hypothesis H1, we have that for any $r>0$,
\begin{eqnarray}\label{eqqq-10}
\lim_{M\rightarrow \infty}\sup_{0<\varepsilon \leq 1}\varepsilon \log P\Big((|u^{r,\varepsilon}|^H_{H^1}(1))^2>M\Big)=-\infty.
\end{eqnarray}
\end{lemma}
Referring to Theorem 4.4 in \cite{H-Z}, we have
\begin{thm}
Let $f \in  C^{1+l}(\mathbb{T}^d)$ for some $l>0$. Under Hypotheses H1 and H3, it holds true that for all $p\in [2, \infty)$,
\begin{eqnarray}\label{eqq-3}
\sup_{r>0}\sup_{0<\varepsilon\leq 1}{E} \sup_{t\in [0,1]}\|\nabla u^{r,\varepsilon}(t)\|^p_{L^{\infty}(\mathbb{T}^d)}< \infty.
\end{eqnarray}
\end{thm}

Now, we devote to giving the proof of  (\ref{eq-8}) under a stronger condition on the initial value $f$. We split it into several lemmas.
Firstly, we aim to prove the following result. Let $u^{r,\varepsilon}_n$ be the solution of (\ref{eqq-2}) with the initial value $f$ replaced by $f_n$. We claim that
\begin{lemma}\label{lemm-3}
Let the initial value $f\in L^p(\mathbb{T}^d)$ for all $p\in [1,\infty)$. Under Hypothesis H1, we have that
for any $\delta>0$ and $r>0$,
\begin{eqnarray}\label{eqqq-8}
\lim_{n\rightarrow +\infty}\sup_{0<\varepsilon \leq 1}\varepsilon \log P\Big(\sup_{0\leq t\leq 1}\|u^{r,\varepsilon}(t)-u^{r,\varepsilon}_n(t)\|^2_H>\delta\Big)=-\infty.
\end{eqnarray}

\end{lemma}
\begin{proof}
For $M>0$, define a stopping time
\[
\tau_{r, \varepsilon,M}:=\inf\Big\{t\geq 0: \|u^{r,\varepsilon}(t)\|^2_H>M, \ \ {\rm{or}}\ \ \varepsilon \varrho \int^t_0\|u^{r,\varepsilon}(s)\|^2_{H^1}ds>M\Big\}.
\]
Clearly,
\begin{eqnarray}\notag
&&P\Big(\sup_{0\leq t\leq 1}\|u^{r,\varepsilon}(t)-u^{r,\varepsilon}_n(t)\|^2_H>\delta, (|u^{r,\varepsilon}|^H_{H^1}(1))^2\leq M\Big)\\ \notag
&\leq& P\Big(\sup_{0\leq t\leq 1}\|u^{r,\varepsilon}(t)-u^{r,\varepsilon}_n(t)\|^2_H>\delta, \tau_{r, \varepsilon, M}\geq 1\Big)\\
\label{eqqq-9}
&\leq& P\Big(\sup_{0\leq t\leq 1\wedge \tau_{r, \varepsilon, M} }\|u^{r,\varepsilon}(t)-u^{r,\varepsilon}_n(t)\|^2_H>\delta\Big).
\end{eqnarray}
Let $k$ be a positive constant will be decided later.
Applying It\^{o} formula to
\[
e^{-k\varepsilon \int^{t\wedge \tau_{r, \varepsilon, M} }_0 \|u^{r,\varepsilon}(s)\|^2_{H^1}ds}\|u^{r,\varepsilon}(t\wedge \tau_{r, \varepsilon, M} )-u^{r,\varepsilon}_n(t\wedge \tau_{r, \varepsilon, M} )\|^2_H,
\]
we get
\begin{eqnarray}\notag
&&e^{-k\varepsilon \int^{t\wedge \tau_{r, \varepsilon, M} }_0 \|u^{r,\varepsilon}(s)\|^2_{H^1}ds}\|u^{r,\varepsilon}(t\wedge \tau_{r, \varepsilon, M} )-u^{r,\varepsilon}_n(t\wedge \tau_{r, \varepsilon, M} )\|^2_H\\ \notag
&=&\|f-f_n\|^2_H-k\varepsilon \int^{t\wedge \tau_{r, \varepsilon, M} }_0e^{-k\varepsilon \int^{s\wedge \tau_{r, \varepsilon, M} }_0 \|u^{r,\varepsilon}(l)\|^2_{H^1}dl}\|u^{r,\varepsilon}(s)\|^2_{H^1}\|u^{r,\varepsilon}(s )-u^{r,\varepsilon}_n(s)\|^2_Hds\\ \notag
&&\ -2\varepsilon \int^{t\wedge \tau_{r, \varepsilon, M} }_0e^{-k\varepsilon \int^{s}_0 \|u^{r,\varepsilon}(l)\|^2_{H^1}dl}\langle u^{r,\varepsilon}(s )-u^{r,\varepsilon}_n(s), div(B(u^{r,\varepsilon}(s))-B(u^{r,\varepsilon}_n(s)))\rangle ds\\ \notag
&&\ +2\varepsilon \int^{t\wedge \tau_{r, \varepsilon, M} }_0e^{-k\varepsilon \int^{s}_0 \|u^{r,\varepsilon}(l)\|^2_{H^1}dl}\langle u^{r,\varepsilon}(s )-u^{r,\varepsilon}_n(s), div(A_r(u^{r,\varepsilon}(s))\nabla u^{r,\varepsilon}(s) -A_r(u^{r,\varepsilon}_n(s))\nabla u^{r,\varepsilon}_n(s))\rangle ds\\ \notag
&&\ +\varepsilon \int^{t\wedge \tau_{r, \varepsilon, M} }_0e^{-k\varepsilon \int^{s}_0 \|u^{r,\varepsilon}(l)\|^2_{H^1}dl}\|\sigma(\varepsilon s, u^{r,\varepsilon}(s))-\sigma(\varepsilon s, u^{r,\varepsilon}_n(s))\|^2_{\mathcal{L}_2(U,H)}ds\\
\notag
&&\ + 2\sqrt{\varepsilon}\int^{t\wedge \tau_{r, \varepsilon, M} }_0e^{-k\varepsilon \int^{s}_0 \|u^{r,\varepsilon}(l)\|^2_{H^1}dl}\langle u^{r,\varepsilon}(s)-u^{r,\varepsilon}_n(s), (\sigma(\varepsilon s, u^{r,\varepsilon}(s))-\sigma(\varepsilon s, u^{r,\varepsilon}_n(s)))dW(s) \rangle\\
\label{eqqq-11}
&:=& \|f-f_n\|^2_H+J_1(t)+J_2(t)+J_3(t)+J_4(t)+J_5(t).
\end{eqnarray}
Utilizing integration by parts, Cauchy-Schwarz inequality and Hypothesis H1, we deduce that
\begin{eqnarray*}
&&-\langle u^{r,\varepsilon}(s )-u^{r,\varepsilon}_n(s), div(B(u^{r,\varepsilon}(s))-B(u^{r,\varepsilon}_n(s)))\rangle\\
&=&\langle \nabla(u^{r,\varepsilon}(s )-u^{r,\varepsilon}_n(s)), B(u^{r,\varepsilon}(s))-B(u^{r,\varepsilon}_n(s))\rangle\\
&\leq& C\|u^{r,\varepsilon}(s)-u^{r,\varepsilon}_n(s)\|_{H^1}\|u^{r,\varepsilon}(s)-u^{r,\varepsilon}_n(s)\|_H\\
&\leq & \frac{\varrho}{2}\|u^{r,\varepsilon}(s)-u^{r,\varepsilon}_n(s)\|^2_{H^1}+\frac{C}{\varrho}\|u^{r,\varepsilon}(s)-u^{r,\varepsilon}_n(s)\|^2_H,
\end{eqnarray*}
which implies that
\begin{eqnarray}\notag
J_2(t)&\leq& \varepsilon  \varrho \int^{t\wedge \tau_{r, \varepsilon, M} }_0e^{-k\varepsilon \int^{s}_0 \|u^{r,\varepsilon}(l)\|^2_{H^1}dl}\|u^{r,\varepsilon}(s)-u^{r,\varepsilon}_n(s)\|^2_{H^1} ds\\
\label{eqqq-12}
&&\ +\frac{\varepsilon C}{\varrho}\int^{t\wedge \tau_{r, \varepsilon, M} }_0e^{-k\varepsilon \int^{s}_0 \|u^{r,\varepsilon}(l)\|^2_{H^1}dl}\|u^{r,\varepsilon}(s)-u^{r,\varepsilon}_n(s)\|^2_Hds.
\end{eqnarray}
By integration by parts, Hypothesis H1, (\ref{eqqq-13}) and Young's inequality, it follows that
\begin{eqnarray*}
&&\langle u^{r,\varepsilon}(s )-u^{r,\varepsilon}_n(s), div(A_r(u^{r,\varepsilon}(s))\nabla u^{r,\varepsilon}(s) -A_r(u^{r,\varepsilon}_n(s))\nabla u^{r,\varepsilon}_n(s))\rangle\\
&=& -\langle \nabla( u^{r,\varepsilon}(s )-u^{r,\varepsilon}_n(s)), A_r(u^{r,\varepsilon}(s))\nabla u^{r,\varepsilon}(s) -A_r(u^{r,\varepsilon}_n(s))\nabla u^{r,\varepsilon}_n(s)\rangle\\
&=& -\langle \nabla( u^{r,\varepsilon}(s )-u^{r,\varepsilon}_n(s)), A_r(u^{r,\varepsilon}_n(s))\nabla (u^{r,\varepsilon}(s)-u^{r,\varepsilon}_n(s)) \rangle\\
&&\ -\langle \nabla( u^{r,\varepsilon}(s )-u^{r,\varepsilon}_n(s)), (A_r(u^{r,\varepsilon}(s))-A_r(u^{r,\varepsilon}_n(s)))\nabla u^{r,\varepsilon}(s) \rangle\\
&\leq& -\varrho \|u^{r,\varepsilon}(s )-u^{r,\varepsilon}_n(s)\|^2_{H^1}+\|A_r(u^{r,\varepsilon}(s))-A_r(u^{r,\varepsilon}_n(s))\|_{L^{\infty}}\|u^{r,\varepsilon}(s )-u^{r,\varepsilon}_n(s)\|_{H^1}\|u^{r,\varepsilon}(s)\|_{H^1}\\
&\leq& -\varrho \|u^{r,\varepsilon}(s )-u^{r,\varepsilon}_n(s)\|^2_{H^1}+C_r\|u^{r,\varepsilon}(s)-u^{r,\varepsilon}_n(s)\|_{H}\|u^{r,\varepsilon}(s )-u^{r,\varepsilon}_n(s)\|_{H^1}\|u^{r,\varepsilon}(s)\|_{H^1}\\
&\leq& -\frac{3\varrho}{4} \|u^{r,\varepsilon}(s )-u^{r,\varepsilon}_n(s)\|^2_{H^1}+\frac{C_r}{\varrho}\|u^{r,\varepsilon}(s)-u^{r,\varepsilon}_n(s)\|^2_{H}\|u^{r,\varepsilon}(s)\|^2_{H^1}.
\end{eqnarray*}
The above inequality yields that
\begin{eqnarray}\notag
J_3(t)&\leq& -\frac{3}{2}\varepsilon  \varrho \int^{t\wedge \tau_{r, \varepsilon, M} }_0e^{-k\varepsilon \int^{s}_0 \|u^{r,\varepsilon}(l)\|^2_{H^1}dl}\|u^{r,\varepsilon}(s)-u^{r,\varepsilon}_n(s)\|^2_{H^1} ds\\
\label{eqqq-15}
&&\ +\frac{2\varepsilon C_r}{\varrho}\int^{t\wedge \tau_{r, \varepsilon, M} }_0e^{-k\varepsilon \int^{s}_0 \|u^{r,\varepsilon}(l)\|^2_{H^1}dl}\|u^{r,\varepsilon}(s)\|^2_{H^1}\|u^{r,\varepsilon}(s)-u^{r,\varepsilon}_n(s)\|^2_Hds.
\end{eqnarray}
By Hypothesis H1, it follows that
\begin{eqnarray}\label{eqqq-16}
J_4(t)\leq C\varepsilon \int^{t\wedge \tau_{r, \varepsilon, M} }_0e^{-k\varepsilon \int^{s}_0 \|u^{r,\varepsilon}(l)\|^2_{H^1}dl}\|u^{r,\varepsilon}(s)- u^{r,\varepsilon}_n(s)\|^2_{H}ds
\end{eqnarray}

Collecting all the above estimates (\ref{eqqq-11})-(\ref{eqqq-16}) and choosing $k>\frac{2C_r}{\varrho}$, we have
\begin{eqnarray}\notag
&&e^{-k\varepsilon \int^{t\wedge \tau_{r, \varepsilon, M} }_0 \|u^{r,\varepsilon}(s)\|^2_{H^1}ds}\|u^{r,\varepsilon}(t\wedge \tau_{r, \varepsilon, M} )-u^{r,\varepsilon}_n(t\wedge \tau_{r, \varepsilon, M} )\|^2_H\\ \notag
&\leq&\|f-f_n\|^2_H+\Big(\frac{\varepsilon C}{\varrho}+C\varepsilon\Big)\int^{t\wedge \tau_{r, \varepsilon, M} }_0e^{-k\varepsilon \int^{s}_0 \|u^{r,\varepsilon}(l)\|^2_{H^1}dl}\|u^{r,\varepsilon}(s)-u^{r,\varepsilon}_n(s)\|^2_Hds+|J_5(t)|.
\end{eqnarray}
With the help of (\ref{eqqq-2}), we get
\begin{eqnarray}\notag
&&\Big(E\Big[\sup_{0\leq s\leq t\wedge \tau_{r,\varepsilon, M}}\Big(e^{-k\varepsilon \int^{s}_0 \|u^{r,\varepsilon}(l)\|^2_{H^1}dl}\|u^{r,\varepsilon}(s )-u^{r,\varepsilon}_n(s )\|^2_H\Big)\Big]^p\Big)^{\frac{2}{p}}\\
\notag
&\leq& 2\|f-f_n\|^4_H+2\Big(\frac{\varepsilon C}{\varrho}+C\varepsilon\Big)^2\int^t_0\Big(E\Big[\Big(\sup_{0\leq l\leq s\wedge \tau_{r,\varepsilon, M}}e^{-k\varepsilon \int^{l}_0 \|u^{r,\varepsilon}(\gamma)\|^2_{H^1}d\gamma}\|u^{r,\varepsilon}(l )-u^{r,\varepsilon}_n(l )\|^2_H\Big)^p\Big]\Big)^{\frac{2}{p}}ds\\
\notag
&&\ +8C\varepsilon p\int^t_0\Big(E\Big[\sup_{0\leq l\leq s\wedge \tau_{r,\varepsilon, M}}\Big(e^{-2k\varepsilon \int^{l}_0 \|u^{r,\varepsilon}(\gamma)\|^2_{H^1}d\gamma}\|u^{r,\varepsilon}(l )-u^{r,\varepsilon}_n(l )\|^4_H\Big)^{\frac{p}{2}}\Big]\Big)^{\frac{2}{p}}ds\\
\notag
&\leq& 2\|f-f_n\|^4_H+2\Big(\frac{\varepsilon C}{\varrho}+C\varepsilon\Big)^2\int^t_0\Big(E\Big[\Big(\sup_{0\leq l\leq s\wedge \tau_{r,\varepsilon, M}}e^{-k\varepsilon \int^{l}_0 \|u^{r,\varepsilon}(\gamma)\|^2_{H^1}d\gamma}\|u^{r,\varepsilon}(l )-u^{r,\varepsilon}_n(l )\|^2_H\Big)^p\Big]\Big)^{\frac{2}{p}}ds\\
\label{eqqq-17}
&&\ +8C\varepsilon p\int^t_0\Big(E\Big[\Big(\sup_{0\leq l\leq s\wedge \tau_{r,\varepsilon, M}}e^{-k\varepsilon \int^{l}_0 \|u^{r,\varepsilon}(\gamma)\|^2_{H^1}d\gamma}\|u^{r,\varepsilon}(l )-u^{r,\varepsilon}_n(l )\|^2_H\Big)^p\Big]\Big)^{\frac{2}{p}}ds.
\end{eqnarray}
Applying Gronwall inequality, we get
\begin{eqnarray}
\Big(E\Big[\sup_{0\leq t\leq 1\wedge \tau_{r,\varepsilon, M}}(e^{-k\varepsilon \int^{t}_0 \|u^{r,\varepsilon}(s)\|^2_{H^1}ds}\|u^{r,\varepsilon}(t)-u^{r,\varepsilon}_n( t)\|^2_H)^p\Big]\Big)^{\frac{2}{p}}\leq 2\|f-f_n\|^4_H e^{2(\varepsilon C+\frac{\varepsilon C}{\delta})^2+8C\varepsilon p}.
\end{eqnarray}
Then, it follows that
\begin{eqnarray}\notag
&&\Big(E\Big[\sup_{0\leq t\leq 1\wedge \tau_{r,\varepsilon, M}}\|u^{r,\varepsilon}(t)-u^{r,\varepsilon}_n( t)\|^{2p}_H\Big]\Big)^{\frac{2}{p}}\\ \notag
&\leq& \Big(E\Big[\sup_{0\leq t\leq 1\wedge \tau_{r,\varepsilon, M}}(e^{-k\varepsilon \int^{t}_0 \|u^{r,\varepsilon}(s)\|^2_{H^1}ds}\|u^{r,\varepsilon}(t)-u^{r,\varepsilon}_n( t)\|^2_H)^pe^{kp\varepsilon \int^{1\wedge \tau_{r,\varepsilon, M}}_0\|u^{r,\varepsilon}(s)\|^2_{H^1}ds}\Big]\Big)^{\frac{2}{p}}\\ \notag
&\leq& e^{2\varepsilon kM}\Big(E\Big[\sup_{0\leq t\leq 1\wedge \tau_{r,\varepsilon, M}}(e^{-k\varepsilon \int^{t}_0 \|u^{r,\varepsilon}(s)\|^2_{H^1}ds}\|u^{r,\varepsilon}(t)-u^{r,\varepsilon}_n( t)\|^2_H)^p\Big]\Big)^{\frac{2}{p}}\\ \label{eqqq-18}
&\leq& 2e^{2\varepsilon kM}\|f-f_n\|^4_H e^{2(\varepsilon C+\frac{\varepsilon C}{\delta})^2+8C\varepsilon p}.
\end{eqnarray}
Fix $M$, and taking $p=\frac{2}{\varepsilon}$ to get
\begin{eqnarray}\notag
&&\sup_{0<\varepsilon\leq 1}\varepsilon \log P\Big(\sup_{0\leq t\leq 1\wedge \tau_{r,\varepsilon, M}}\|u^{r,\varepsilon}(t)-u^{r,\varepsilon}_n( t)\|^2_H> \delta\Big)\\ \notag
&\leq& \sup_{0<\varepsilon\leq 1}\varepsilon \log\frac{E\Big[\sup_{0\leq t\leq 1\wedge \tau_{r,\varepsilon, M}}\|u^{r,\varepsilon}(t)-u^{r,\varepsilon}_n( t)\|^{2p}_H\Big]}{\delta^p}\\ \notag
&\leq& 2kM+2(C+\frac{ C}{\delta})^2+8C-2\log \delta+\log 2 \|f-f_n\|^4_H\\
\label{eqqq-19}
&\rightarrow & -\infty, \quad {\rm{as}}\ n\rightarrow +\infty.
\end{eqnarray}
By Lemma \ref{lemm-2}, for any $R>0$, there exists a constant $M$ such that for any $\varepsilon\in (0,1]$ and $r>0 $, the following inequality holds,
\begin{eqnarray}\label{eqqq-20}
P\Big((|u^{r,\varepsilon}|^H_{H^1}(1))^2>M\Big)\leq e^{-\frac{R}{\varepsilon}}.
\end{eqnarray}
For such a $M$, by (\ref{eqqq-9}) and (\ref{eqqq-19}), there exists a positive integer $N$, such that for any $n\geq N$,
\begin{eqnarray}\label{eqqq-21}
\sup_{0<\varepsilon\leq 1}\varepsilon \log P\Big(\sup_{0\leq t\leq 1}\|u^{r,\varepsilon}(t)-u^{r,\varepsilon}_n( t)\|^2_H> \delta,(|u^{r,\varepsilon}|^H_{H^1}(1))^2\leq M \Big)\leq -R.
\end{eqnarray}
Putting (\ref{eqqq-20}) and (\ref{eqqq-21}) together, one sees that there exists a positive integer $N$, such that for any $n\geq N$, $\varepsilon\in (0,1]$ and $r>0$,
\begin{eqnarray}\label{eqqq-22}
 P\Big(\sup_{0\leq t\leq 1}\|u^{r,\varepsilon}(t)-u^{r,\varepsilon}_n( t)\|^2_H> \delta\Big)\leq 2e^{-\frac{R}{\varepsilon}}.
\end{eqnarray}
Since $R$ is arbitrary, the lemma follows.
\end{proof}
Using similar argument as Lemma 3.4 in \cite{X-Z}, we have the following  result for the difference $v^{\varepsilon}(\cdot)-v^{\varepsilon}_n(\cdot)$, where $v^{\varepsilon}_n$ is the solution of (\ref{eq-7}) with the initial value $f_n$.
\begin{lemma}\label{lemm-4}
Let the initial value $f\in L^p(\mathbb{T}^d)$ for all $p\in [1,\infty)$. Under Hypothesis H1, we have that
for any $\delta>0$,
\begin{eqnarray}
\lim_{n\rightarrow +\infty}\sup_{0<\varepsilon\leq 1}\varepsilon \log P\Big(\sup_{0\leq t\leq 1}\|v^{\varepsilon}(t)-v^{\varepsilon}_n(t)\|^2_H>\delta\Big)=-\infty.
\end{eqnarray}
\end{lemma}

\begin{lemma}\label{lemm-5}
Let the initial value $f\in L^p(\mathbb{T}^d)$ for all $p\in [1,\infty)$. Under Hypotheses H1-H2,
we have that for any $\delta>0$ and $r>0$, and every positive integer $n$,
\begin{eqnarray}
\lim_{\varepsilon\rightarrow 0}\varepsilon \log P\Big(\sup_{0\leq t\leq 1}\|u^{r,\varepsilon}_n(t)-v^{\varepsilon}_n(t)\|^2_H>\delta\Big)=-\infty.
\end{eqnarray}

\end{lemma}
\begin{proof}
For any $M>0$, define the following stopping times
\begin{eqnarray*}
\tau^{1}_{r,\varepsilon, M}&:=&\inf\Big\{t\geq 0: \varepsilon \varrho \int^t_0\|u^{r,\varepsilon}_n(s)\|^2_{H^1}ds>M, \ {\rm{or}}\ \|u^{r,\varepsilon}_n(t)\|^2_H>M \Big\},\\
\tau^{2}_{\varepsilon, M}&:=&\inf\Big\{t\geq 0: \|v^{\varepsilon}_n(t)\|^2_{H^1}>M \Big\},\\
\tau&:=&\tau^{1}_{r,\varepsilon, M}\wedge\tau^{2}_{\varepsilon, M} .
\end{eqnarray*}
Then we have
\begin{eqnarray*}
&&P\Big(\sup_{0\leq t\leq 1}\|u^{r,\varepsilon}_n(t)-v^{\varepsilon}_n(t)\|^2_H>\delta, (|u^{r,\varepsilon}_n|^H_{H^1}(1))^2\leq M, \sup_{t\in[0,1]}\|v^{\varepsilon}_n(t)\|^2_{H^1}\leq M\Big)\\
&\leq& P\Big(\sup_{0\leq t\leq 1}\|u^{r,\varepsilon}_n(t)-v^{\varepsilon}_n(t)\|^2_H>\delta, \tau\geq 1\Big)\\
&\leq& P\Big(\sup_{0\leq t\leq 1\wedge \tau}\|u^{r,\varepsilon}_n(t)-v^{\varepsilon}_n(t)\|^2_H>\delta\Big).
\end{eqnarray*}
Applying It\^{o} formula to $\|u^{r,\varepsilon}_n(t\wedge \tau)-v^{\varepsilon}_n(t\wedge \tau)\|^2_H$, we deduce that
\begin{eqnarray*}
&&\|u^{r,\varepsilon}_n(t\wedge \tau)-v^{\varepsilon}_n(t\wedge \tau)\|^2_H\\
&=& -2\varepsilon \int^{t\wedge \tau}_0\langle u^{r,\varepsilon}_n(s)-v^{\varepsilon}_n(s), div(B(u^{r,\varepsilon}_n(s)))\rangle ds\\
&&\ +2\varepsilon \int^{t\wedge \tau}_0\langle u^{r,\varepsilon}_n(s)-v^{\varepsilon}_n(s), div(A_r(u^{r,\varepsilon}_n(s))\nabla u^{r,\varepsilon}_n(s))\rangle ds\\
&&\ +2\sqrt{\varepsilon} \int^{t\wedge \tau}_0\langle u^{r,\varepsilon}_n(s)-v^{\varepsilon}_n(s), (\sigma(\varepsilon s,u^{r,\varepsilon}_n(s))-\sigma(\varepsilon s,v^{\varepsilon}_n(s))dW(s) \rangle\\
&&\ +\varepsilon \int^{t\wedge \tau}_0\|\sigma(\varepsilon s,u^{r,\varepsilon}_n(s))-\sigma(\varepsilon s,v^{\varepsilon}_n(s)\|^2_{\mathcal{L}_2(U,H)}ds\\
&:=&K_1(t)+K_2(t)+K_3(t)+K_4(t).
\end{eqnarray*}
By integration by parts, Hypothesis H1 and Young's inequality, we get
\begin{eqnarray*}
K_1(t)&\leq& 2\varepsilon C \int^{t\wedge \tau}_0\| u^{r,\varepsilon}_n(s)-v^{\varepsilon}_n(s)\|_{H^1} (1+\|u^{r,\varepsilon}_n(s)\|_{H}) ds\\
&\leq& \frac{\varepsilon \varrho}{2}\int^{t\wedge \tau}_0\| u^{r,\varepsilon}_n(s)-v^{\varepsilon}_n(s)\|^2_{H^1}ds+\frac{\varepsilon C}{\varrho}\int^{t\wedge \tau}_0(1+\|u^{r,\varepsilon}_n(s)\|^2_{H})ds.
\end{eqnarray*}
By integration by parts, Hypothesis H1, (\ref{eqqq-13}) and Young's inequality, it follows that
\begin{eqnarray*}
K_2(t)&=&-2\varepsilon \int^{t\wedge \tau}_0\langle \nabla(u^{r,\varepsilon}_n(s)-v^{\varepsilon}_n(s)), A_r(u^{r,\varepsilon}_n(s))\nabla (u^{r,\varepsilon}_n(s)-v^{\varepsilon}_n(s))\rangle ds\\
&&\ -2\varepsilon \int^{t\wedge \tau}_0\langle \nabla(u^{r,\varepsilon}_n(s)-v^{\varepsilon}_n(s)), A_r(u^{r,\varepsilon}_n(s))\nabla v^{\varepsilon}_n(s)\rangle ds\\
&\leq& -2\varepsilon \varrho \int^{t\wedge \tau}_0\| u^{r,\varepsilon}_n(s)-v^{\varepsilon}_n(s)\|^2_{H^1}ds
+2\varepsilon \int^{t\wedge \tau}_0 \|A_r(u^{r,\varepsilon}_n(s))\|_{L^{\infty}}\| u^{r,\varepsilon}_n(s)-v^{\varepsilon}_n(s)\|_{H^1}\|v^{\varepsilon}_n(s)\|_{H^1}ds\\
&\leq& -2\varepsilon \varrho \int^{t\wedge \tau}_0\| u^{r,\varepsilon}_n(s)-v^{\varepsilon}_n(s)\|^2_{H^1}ds
+2\varepsilon C_r \int^{t\wedge \tau}_0 \|u^{r,\varepsilon}_n(s)\|_{H}\| u^{r,\varepsilon}_n(s)-v^{\varepsilon}_n(s)\|_{H^1}\|v^{\varepsilon}_n(s)\|_{H^1}ds\\
&\leq& -\frac{3}{2}\varepsilon \varrho \int^{t\wedge \tau}_0\| u^{r,\varepsilon}_n(s)-v^{\varepsilon}_n(s)\|^2_{H^1}ds
+\frac{\varepsilon C_r}{\varrho} \int^{t\wedge \tau}_0 \|u^{r,\varepsilon}_n(s)\|^2_{H}\|v^{\varepsilon}_n(s)\|^2_{H^1}ds.
\end{eqnarray*}
Utilizing Hypothesis H1, it follows that
\begin{eqnarray*}
K_4(t)\leq \varepsilon C \int^{t\wedge \tau}_0\|u^{r,\varepsilon}_n(s)-v^{\varepsilon}_n(s)\|^2_{H}ds.
\end{eqnarray*}
Collecting all the previous estimates, we deduce that
\begin{eqnarray*}
&&\|u^{r,\varepsilon}_n(t\wedge \tau)-v^{\varepsilon}_n(t\wedge \tau)\|^2_H+\varepsilon \varrho \int^{t\wedge \tau}_0\| u^{r,\varepsilon}_n(s)-v^{\varepsilon}_n(s)\|^2_{H^1}ds\\
&\leq & \frac{\varepsilon C}{\varrho}\int^{t\wedge \tau}_0(1+\|u^{r,\varepsilon}_n(s)\|^2_{H})ds+\frac{\varepsilon C_r}{\varrho} \int^{t\wedge \tau}_0 \|u^{r,\varepsilon}_n(s)\|^2_{H}\|v^{\varepsilon}_n(s)\|^2_{H^1}ds\\
&&\ +\varepsilon C \int^{t\wedge \tau}_0\|u^{r,\varepsilon}_n(s)-v^{\varepsilon}_n(s)\|^2_{H}ds+|K_3(t)|.
\end{eqnarray*}
Using Gronwall's inequality, we get
\begin{eqnarray*}
&&\|u^{r,\varepsilon}_n(t\wedge \tau)-v^{\varepsilon}_n(t\wedge \tau)\|^2_H\\
&\leq& \Big[\frac{\varepsilon C}{\varrho}\int^{t\wedge \tau}_0(1+\|u^{r,\varepsilon}_n(s)\|^2_{H})ds+\frac{\varepsilon C_r}{\varrho} \int^{t\wedge \tau}_0 \|u^{r,\varepsilon}_n(s)\|^2_{H}\|v^{\varepsilon}_n(s)\|^2_{H^1}ds+|K_3(t)|\Big]e^{\varepsilon C t}.
\end{eqnarray*}
By (\ref{eqqq-2}) and the definition of $\tau$, we deduce that
\begin{eqnarray*}
&&\Big(E \sup_{0\leq s\leq t\wedge \tau}\|u^{r,\varepsilon}_n(s)-v^{\varepsilon}_n(s)\|^{2p}_H\Big)^{\frac{2}{p}}\\
&\leq& e^{\varepsilon C}\Big[\frac{\varepsilon^2 C_r}{\varrho^2}\Big(E\Big(\int^{t\wedge  \tau}_0 \|u^{r,\varepsilon}_n(s)\|^2_{H}\|v^{\varepsilon}_n(s)\|^2_{H^1}ds\Big)^p\Big)^{\frac{2}{p}}+\frac{\varepsilon^2 C}{\varrho^2} \Big(E\Big(\int^{t\wedge  \tau}_0 (1+\|u^{r,\varepsilon}_n(s)\|^2_{H})ds\Big)^p\Big)^{\frac{2}{p}}\\
&&\ +8\varepsilon pC\int^t_0\Big(E\sup_{0\leq l\leq s\wedge \tau}\|u^{r,\varepsilon}_n(l)-v^{\varepsilon}_n(l)\|^{2p}_H\Big)^{\frac{2}{p}}ds\Big]\\
&\leq & e^{\varepsilon C}\Big[\frac{\varepsilon^2 C_r}{\varrho^2}M^4+\frac{\varepsilon^2 C}{\varrho^2}(1+M)^2\Big]
+e^{\varepsilon C}\cdot 8\varepsilon pC\int^t_0\Big(E\sup_{0\leq l\leq s\wedge \tau}\|u^{r,\varepsilon}_n(l)-v^{\varepsilon}_n(l)\|^{2p}_H\Big)^{\frac{2}{p}}ds.
\end{eqnarray*}
Applying Gronwall inequality again, we get
\begin{eqnarray}\notag
&&\Big(E \sup_{0\leq s\leq t\wedge \tau}\|u^{r,\varepsilon}_n(s)-v^{\varepsilon}_n(s)\|^{2p}_H\Big)^{\frac{2}{p}}\\
\label{eqqq-23}
&\leq& e^{\varepsilon C}\Big[\frac{\varepsilon^2 C_r}{\varrho^2}M^4+\frac{\varepsilon^2 C}{\varrho^2}(1+M)^2\Big]\cdot e^{8\varepsilon pC e^{\varepsilon C}}.
\end{eqnarray}
From (\ref{e-4}) and Lemma \ref{lemm-7}, we know that for any $R>0$, there exits a $M$ such that
\begin{eqnarray}\label{eqqq-24}
\sup_{0<\varepsilon \leq 1}\varepsilon \log P\Big((|u^{r,\varepsilon}_n|^H_{H^1}(1))^2>M\Big)&\leq &-R,\\
\label{eqqq-25}
\sup_{0<\varepsilon \leq 1}\varepsilon \log P\Big(\sup_{t\in [0,1]}\|v^{\varepsilon}_n(t)\|^2_{H^1}>M\Big)&\leq & -R.
\end{eqnarray}

For such a constant $M$, let $p=\frac{2}{\varepsilon}$ in (\ref{eqqq-23}) to obtain
\begin{eqnarray*}
&&\varepsilon \log P\Big(\sup_{0\leq s\leq 1 }\|u^{r,\varepsilon}_n(s)-v^{\varepsilon}_n(s)\|^2_H>\delta, (|u^{r,\varepsilon}_n|^H_{H^1}(1))^2\leq M,\sup_{t\in [0,1]}\|v^{\varepsilon}_n(t)\|^2_{H^1}\leq M\Big)\\
&\leq& \varepsilon \log P\Big(\sup_{0\leq s\leq 1\wedge \tau }\|u^{r,\varepsilon}_n(s)-v^{\varepsilon}_n(s)\|^2_H>\delta\Big)\\
&\leq & \log \Big(E\sup_{0\leq s\leq 1\wedge \tau }\|u^{r,\varepsilon}_n(s)-v^{\varepsilon}_n(s)\|^{2p}_H\Big)^{\frac{2}{p}}-2\log \delta\\
&\leq& 2\varepsilon C+\log\Big[\frac{\varepsilon^2 C}{\varrho^2}M^4+\frac{\varepsilon^2 C_r}{\varrho^2}(1+M)^2\Big]+8\varepsilon pC e^{\varepsilon C}-2\log \delta\\
&\rightarrow& -\infty, \quad {\rm{as}}\ \varepsilon \rightarrow 0.
\end{eqnarray*}
Thus, for any $R>0$, there exists a $\varepsilon_0$ such that for any $\varepsilon$ satisfying $0<\varepsilon\leq \varepsilon_0$,
\begin{eqnarray}\label{eqqq-26}
P\Big(\sup_{0\leq s\leq 1 }\|u^{r,\varepsilon}_n(s)-v^{\varepsilon}_n(s)\|^2_H>\delta, (|u^{r,\varepsilon}_n|^H_{H^1}(1))^2\leq M,\sup_{t\in [0,1]}\|v^{\varepsilon}_n(t)\|^2_{H^1}\leq M\Big)\leq e^{-\frac{R}{\varepsilon}}.
\end{eqnarray}
Putting (\ref{eqqq-24})-(\ref{eqqq-26}) together, we see that there exists a constant $\varepsilon_0$ such that for any $\varepsilon\in (0, \varepsilon_0]$, $r>0$ and $n\geq 1$,
\begin{eqnarray*}
P\Big(\sup_{0\leq s\leq 1 }\|u^{r,\varepsilon}_n(s)-v^{\varepsilon}_n(s)\|^2_H>\delta\Big)\leq 3e^{-\frac{R}{\varepsilon}}.
\end{eqnarray*}
Since $R$ is arbitrary, the proof is completed.

\end{proof}

\begin{lemma} \label{lemm-6}
Let $f \in  C^{1+l}(\mathbb{T}^d)$ for some $l>0$. Under Hypotheses H1 and H3,
we have that for any $\delta>0$,
\begin{eqnarray}
\lim_{r\rightarrow 0}\sup_{0<\varepsilon\leq 1}\varepsilon \log P\Big(\sup_{0\leq t\leq 1}\|u^{\varepsilon}(t)-u^{r,\varepsilon}(t)\|^2_H>\delta\Big)=-\infty.
\end{eqnarray}
\end{lemma}
\begin{proof}
From (\ref{eqq-5}) and (\ref{eqq-2}), we deduce that
\begin{eqnarray*}
&&u^{\varepsilon}(t)-u^{r,\varepsilon}(t)+\varepsilon\int^t_0 div(B(u^{\varepsilon}(s))-B(u^{r,\varepsilon}(s)))ds\\
&=& \varepsilon \int^t_0  div(A(u^{\varepsilon}(s))\nabla u^{\varepsilon}(s)-A_r(u^{r,\varepsilon}(s))\nabla u^{r,\varepsilon}(s))ds\\
&&\ +\sqrt{\varepsilon} \int^t_0 (\sigma(\varepsilon s,u^{\varepsilon}(s) )-\sigma(\varepsilon s,u^{r,\varepsilon}(s) ))dW(s).
\end{eqnarray*}
Applying It\^{o} formula to $\|u^{\varepsilon}(t)-u^{r,\varepsilon}(t)\|^2_H$, one obtains that
\begin{eqnarray*}
&&\|u^{\varepsilon}(t)-u^{r,\varepsilon}(t)\|^2_H\\
&=& -2\varepsilon \int^{t}_0 \langle u^{\varepsilon}(s)-u^{r,\varepsilon}(s), div ( B(u^{\varepsilon}(s))-B(u^{r,\varepsilon}(s)))\rangle ds\\
&&\ +2\varepsilon \int^{t}_0 \langle u^{\varepsilon}(s)-u^{r,\varepsilon}(s), div ( A(u^{\varepsilon}(s))\nabla u^{\varepsilon}(s) -A(u^{r,\varepsilon}(s))\nabla u^{r,\varepsilon}(s))\rangle ds\\
&&\ +2\sqrt{\varepsilon} \int^{t}_0\langle u^{\varepsilon}(s)-u^{r,\varepsilon}(s), (\sigma(\varepsilon s,u^{\varepsilon}(s))-\sigma(\varepsilon s,u^{r,\varepsilon}(s))dW(s) \rangle\\
&&\ +\varepsilon \int^{t}_0\|\sigma(\varepsilon s,u^{\varepsilon}(s))-\sigma(\varepsilon s,u^{r,\varepsilon}(s)\|^2_{\mathcal{L}_2(U,H)}ds\\
&:=&L_1(t)+L_2(t)+L_3(t)+L_4(t).
\end{eqnarray*}
By integration by parts, Hypothesis H1 and Young's inequality, we have
\begin{eqnarray*}
L_1(t)&=&2\varepsilon \int^{t}_0 \langle \nabla(u^{\varepsilon}(s)-u^{r,\varepsilon}(s)),  B(u^{\varepsilon}(s))-B(u^{r,\varepsilon}(s))\rangle ds\\
&\leq& 2\varepsilon C\int^{t}_0 \|u^{\varepsilon}(s)-u^{r,\varepsilon}(s)\|_{H^1}\|u^{\varepsilon}(s)-u^{r,\varepsilon}(s)\|_Hds\\
&\leq& \frac{\varepsilon \varrho }{2}\int^{t}_0 \|u^{\varepsilon}(s)-u^{r,\varepsilon}(s)\|^2_{H^1}ds+\frac{\varepsilon C}{\varrho}\int^{t}_0 \|u^{\varepsilon}(s)-u^{r,\varepsilon}(s)\|^2_{H}ds.
\end{eqnarray*}
By integration by parts, we deduce that
\begin{eqnarray}\notag
L_2(t)&=&-2\varepsilon \int^{t}_0 \langle \nabla (u^{\varepsilon}(s)-u^{r,\varepsilon}(s)), A(u^{\varepsilon}(s))\nabla u^{\varepsilon}(s) -A_r(u^{r,\varepsilon}(s))\nabla u^{r,\varepsilon}(s)\rangle ds\\ \notag
&=& -2\varepsilon \int^{t}_0 \langle
\nabla (u^{\varepsilon}(s)-u^{r,\varepsilon}(s)),A(u^{\varepsilon}(s))\nabla (u^{\varepsilon}(s)-u^{r,\varepsilon}(s))\rangle ds\\ \notag
&&\ -2\varepsilon \int^{t}_0 \langle \nabla (u^{\varepsilon}(s)-u^{r,\varepsilon}(s)), (A(u^{\varepsilon}(s)) -A_r(u^{\varepsilon}(s)))\nabla u^{r,\varepsilon}(s)\rangle ds\\ \notag
&&\ -2\varepsilon \int^{t}_0 \langle \nabla (u^{\varepsilon}(s)-u^{r,\varepsilon}(s)), (A_r(u^{\varepsilon}(s))-A_r(u^{r,\varepsilon}(s)))\nabla u^{r,\varepsilon}(s)\rangle ds\\
\label{eqqq-31}
&\leq& -2\varepsilon \varrho \int^{t}_0 \| u^{\varepsilon}(s)-u^{r,\varepsilon}(s)\|^2_{H^1} ds+L_{2,1}(t)+L_{2,2}(t).
\end{eqnarray}
For $L_{2,1}(t)$, by Young's inequality, we reach that
\begin{eqnarray}\notag
L_{2,1}(t)&\leq& 2\varepsilon \int^{t}_0 \|\nabla u^{r,\varepsilon}(s)\|_{L^{\infty}}\left(\int_{\mathbb{T}^d} |\nabla (u^{\varepsilon}(s)-u^{r,\varepsilon}(s))||A(u^{\varepsilon}(s))-P_r (A(u^{\varepsilon}(s)))|dx\right)ds\\ \notag
&\leq& 2\varepsilon \int^{t}_0 \|\nabla u^{r,\varepsilon}(s)\|_{L^{\infty}}\|u^{\varepsilon}(s)-u^{r,\varepsilon}(s)\|_{H^1}\|A(u^{\varepsilon}(s))-P_r (A(u^{\varepsilon}(s)))\|_Hds\\ \notag
&\leq&\frac{\varepsilon \varrho}{2} \int^{t}_0 \| u^{\varepsilon}(s)-u^{r,\varepsilon}(s)\|^2_{H^1}ds+\frac{\varepsilon C}{\varrho}\int^{t}_0\|\nabla u^{r,\varepsilon}(s)\|^2_{L^{\infty}}|A(u^{\varepsilon}(s))-P_r A(u^{\varepsilon}(s))\|^2_Hds\\
\notag
&\leq&\frac{\varepsilon \varrho}{2} \int^{t}_0 \| u^{\varepsilon}(s)-u^{r,\varepsilon}(s)\|^2_{H^1}ds\\
\label{eqqq-36}
&&\ +\frac{\varepsilon C}{\varrho}\sup_{0\leq s\leq t}\|\nabla u^{r,\varepsilon}(s)\|^2_{L^{\infty}}\int^{t}_0|A(u^{\varepsilon}(s))-P_r A(u^{\varepsilon}(s))\|^2_Hds.
\end{eqnarray}

Utilizing the contraction property of the semigroup $P_r$, Cauchy-Schwarz inequality and Young's inequality, we get
\begin{eqnarray}\notag
L_{2,2}(t)&\leq& 2\varepsilon \int^{t}_0
\|A_r(u^{\varepsilon}(s))-A_r(u^{r,\varepsilon}(s))\|_{H}
\|u^{\varepsilon}(s)-u^{r,\varepsilon}(s)\|_{H^1} \|\nabla u^{r,\varepsilon}(s)\|_{L^{\infty}}ds\\ \notag
&\leq& 2\varepsilon C \int^{t}_0 \|u^{\varepsilon}(s)-u^{r,\varepsilon}(s)\|_H\|u^{\varepsilon}(s)-u^{r,\varepsilon}(s)\|_{H^1} \|\nabla u^{r,\varepsilon}(s)\|_{L^{\infty}}ds\\
\label{eqqq-33}
&\leq& \frac{\varepsilon \varrho }{2}\int^{t}_0 \|u^{\varepsilon}(s)-u^{r,\varepsilon}(s)\|^2_{H^1}ds+\frac{\varepsilon C}{\varrho}\int^{t}_0  \|\nabla u^{r,\varepsilon}(s)\|^2_{L^{\infty}}\|u^{\varepsilon}(s)-u^{r,\varepsilon}(s)\|^2_Hds.
\end{eqnarray}
Combing (\ref{eqqq-31})-(\ref{eqqq-33}), it follows that
\begin{eqnarray*}
L_2(t)&\leq& -\varepsilon \varrho \int^{t}_0 \| u^{\varepsilon}(s)-u^{r,\varepsilon}(s)\|^2_{H^1} ds+\frac{\varepsilon C}{\varrho}\sup_{0\leq s\leq t}\|\nabla u^{r,\varepsilon}(s)\|^2_{L^{\infty}}\int^{t}_0|A(u^{\varepsilon}(s))-P_r A(u^{\varepsilon}(s))\|^2_Hds\\
&&\ +\frac{\varepsilon C}{\varrho}\int^{t}_0  \|\nabla u^{r,\varepsilon}(s)\|^2_{L^{\infty}}\|u^{\varepsilon}(s)-u^{r,\varepsilon}(s)\|^2_Hds.
\end{eqnarray*}
Moreover, by Hypothesis H1, we obtain
\begin{eqnarray*}
L_4(t)\leq \varepsilon C\int^{t}_0\|u^{\varepsilon}(s)-u^{r,\varepsilon}(s)\|^2_{H}ds.
\end{eqnarray*}
Combing all the previous estimates, we deduce that
\begin{eqnarray*}
&&\|u^{\varepsilon}(t)-u^{r,\varepsilon}(t)\|^2_H\\
&\leq& \frac{\varepsilon C}{\varrho}\sup_{0\leq s\leq t}\|\nabla u^{r,\varepsilon}(s)\|^2_{L^{\infty}}\int^{t}_0|A(u^{\varepsilon}(s))-P_r A(u^{\varepsilon}(s))\|^2_Hds\\
&&\ +\int^{t}_0 \Big( \frac{\varepsilon C}{\varrho}\|\nabla u^{r,\varepsilon}(s)\|^2_{L^{\infty}}+ \varepsilon C+\frac{\varepsilon C}{\varrho}\Big) \|u^{\varepsilon}(s)-u^{r,\varepsilon}(s)\|^2_Hds+|L_3(t)|.
\end{eqnarray*}
By using Gronwall inequality, we get
\begin{eqnarray*}
&&\|u^{\varepsilon}(t)-u^{r,\varepsilon}(t)\|^2_H\\
&\leq& \Big[\frac{\varepsilon C}{\varrho}\sup_{0\leq s\leq t}\|\nabla u^{r,\varepsilon}(s)\|^2_{L^{\infty}}\int^{t}_0\|A(u^{\varepsilon}(s))-P_r A(u^{\varepsilon}(s))\|^2_Hds+|L_3(t)|\Big]\\
 &&\ \times \exp\Big\{\frac{\varepsilon C}{\varrho} t+\frac{ C }{\varrho^2}\sup_{0\leq s\leq t}\|\nabla u^{r,\varepsilon}(s)\|^2_{L^{\infty}}+ \varepsilon C t\Big\}.
\end{eqnarray*}
Using the similar arguments as (\ref{eqqq-3}), we deduce that
\begin{eqnarray*}
&&\Big(E\sup_{0\leq s\leq 1}\|u^{\varepsilon}(s)-u^{r,\varepsilon}(s)\|^{2p}_H\Big)^{\frac{2}{p}}\\
&\leq& \Big[\frac{\varepsilon^2 C}{\varrho^2}\Big(E(\sup_{0\leq s\leq 1}\|\nabla u^{r,\varepsilon}(s)\|^{2p}_{L^{\infty}})\Big)^{\frac{2}{p}}\Big(E(\int^{1}_0\|A(u^{\varepsilon}(s))-P_r A(u^{\varepsilon}(s))\|^2_Hds)^{2p}\Big)^{\frac{2}{p}}\Big]\\
&&\ \times \exp\Big\{\frac{2\varepsilon C}{\varrho} +\frac{ 2C }{\varrho^2}\Big(E\sup_{0\leq s\leq 1}\|\nabla u^{r,\varepsilon}(s)\|^{2p}_{L^{\infty}}\Big)^{\frac{2}{p}}+ 2\varepsilon C \Big\} \\
&&\ +8C\varepsilon p\exp\Big\{\frac{2\varepsilon C}{\varrho} +\frac{2 C }{\varrho^2}\Big(E\sup_{0\leq s\leq 1}\|\nabla u^{r,\varepsilon}(s)\|^{2p}_{L^{\infty}}\Big)^{\frac{2}{p}}+ 2\varepsilon C\Big\} \int^1_0\Big(E\sup_{0\leq l\leq s}\|u^{\varepsilon}(l)-u^{r,\varepsilon}(l)\|^{2p}_H\Big)^{\frac{2}{p}}ds.
\end{eqnarray*}
Applying Gronwall inequality, we obtain
\begin{eqnarray}\notag
&&\Big(E\sup_{0\leq s\leq 1}\|u^{\varepsilon}(s)-u^{r,\varepsilon}(s)\|^{2p}_H\Big)^{\frac{2}{p}}\\
\notag
&\leq&\Big[\frac{\varepsilon^2 C}{\varrho^2}\Big(E\Big(\sup_{0\leq s\leq 1}\|\nabla u^{r,\varepsilon}(s)\|^{2p}_{L^{\infty}}\Big)\Big)^{\frac{2}{p}}\Big(E\Big(\int^{1}_0\|A(u^{\varepsilon}(s))-P_r A(u^{\varepsilon}(s))\|^2_Hds\Big)^{2p}\Big)^{\frac{2}{p}}\Big]\\
\label{eqqq-28}
&&\ \times \exp\Big\{\frac{2\varepsilon C}{\varrho} +\frac{ 2C }{\varrho^2}\Big(E\sup_{0\leq s\leq 1}\|\nabla u^{r,\varepsilon}(s)\|^{2p}_{L^{\infty}}\Big)^{\frac{2}{p}}+ 2\varepsilon C \Big\}e^{C(\varepsilon, p,  \varrho)},
\end{eqnarray}
where
\begin{eqnarray}\label{eqqq-42}
C(\varepsilon, p, \varrho)=8C\varepsilon p\exp\Big\{\frac{2\varepsilon C}{\varrho} +\frac{2\varepsilon C}{\varrho}\Big(E\sup_{0\leq s\leq 1}\|\nabla u^{r,\varepsilon}(s)\|^{2p}_{L^{\infty}}\Big)^{\frac{2}{p}}+ 2\varepsilon C \Big\}.
\end{eqnarray}

Let $p=\frac{2}{\varepsilon}$ in (\ref{eqqq-28}) to obtain
\begin{eqnarray*}
&&\varepsilon \log P\Big(\sup_{0\leq s\leq 1}\|u^{\varepsilon}(s)-u^{r,\varepsilon}(s)\|^2_H>\delta \Big)\\ \notag
&\leq & \log \Big(E\sup_{0\leq s\leq 1}\|u^{\varepsilon}(s)-u^{r,\varepsilon}(s)\|^{2p}_H\Big)^{\frac{2}{p}}-2\log \delta\\ \notag
&\leq& \frac{2 \varepsilon C}{\varrho} +\frac{ C }{\varrho^2}\Big(E\sup_{0\leq s\leq 1}\|\nabla u^{r,\varepsilon}(s)\|^{2p}_{L^{\infty}}\Big)^{\frac{2}{p}}+ \varepsilon C\\
&&\ +\log\Big[\frac{\varepsilon^2 C}{\varrho^2}\Big(E\sup_{0\leq s\leq 1}\|\nabla u^{r,\varepsilon}(s)\|^{4p}_{L^{\infty}}\Big)^{\frac{1}{p}}\Big(E(\int^{1}_0\|A(u^{\varepsilon}(s))-P_r A(u^{\varepsilon}(s))\|^2_Hds)^{2p}\Big)^{\frac{1}{p}}\Big]\\ \notag
&&\ +C(\varepsilon, p,  \varrho)-2\log \delta,
\end{eqnarray*}
which implies that
\begin{eqnarray}\notag
&&\sup_{0<\varepsilon\leq 1}\varepsilon \log P\Big(\sup_{0\leq s\leq 1}\|u^{\varepsilon}(s)-u^{r,\varepsilon}(s)\|^2_H>\delta \Big)\\ \notag
&\leq& \log\Big[\frac{ C}{\varrho^2} \sup_{0<\varepsilon\leq 1}\Big(E\sup_{0\leq s\leq 1}\|\nabla u^{r,\varepsilon}(s)\|^{2p}_{L^{\infty}}\Big)^{\frac{2}{p}}\Big(E\Big(\int^{1}_0\|A(u^{\varepsilon}(s))-P_r A(u^{\varepsilon}(s))\|^2_Hds\Big)^{2p}\Big)^{\frac{2}{p}}\Big]\\
\label{eqqq-30}
&&\ +\frac{2 C}{\varrho} +\frac{ C }{\varrho^2}\sup_{0<\varepsilon\leq 1}\Big(E\sup_{0\leq s\leq 1}\|\nabla u^{r,\varepsilon}(s)\|^{2p}_{L^{\infty}}\Big)^{\frac{2}{p}}+ C+\sup_{0<\varepsilon\leq 1}C(\varepsilon, p, \varrho)-2\log \delta.
\end{eqnarray}
According to (\ref{eqq-3}) and by (\ref{eqqq-42}), it follows that
\begin{eqnarray}\label{eqqq-40}
\sup_{0<\varepsilon\leq 1}\Big(E\sup_{0\leq s\leq 1}\|\nabla u^{r,\varepsilon}(s)\|^{2p}_{L^{\infty}}\Big)^{\frac{2}{p}}<\infty, \quad {\rm{and}}\  \sup_{0<\varepsilon\leq 1}C(\varepsilon, p, \varrho)<\infty.
\end{eqnarray}
By the strong continuity of the semigroup $P_r$ and the boundness of $A(u^{\varepsilon})$, we have
\begin{eqnarray}\label{eqqq-41}
\Big(E\Big(\int^{1}_0\|A(u^{\varepsilon}(s))-P_r A(u^{\varepsilon}(s))\|^2_Hds\Big)^{2p}\Big)^{\frac{2}{p}}\rightarrow 0, \quad r\rightarrow 0.
\end{eqnarray}
Combing (\ref{eqqq-30})-(\ref{eqqq-41}), we deduce that
\begin{eqnarray*}\notag
\lim_{r\rightarrow 0}\sup_{0<\varepsilon\leq 1}\varepsilon \log P\Big(\sup_{0\leq s\leq 1}\|u^{\varepsilon}(s)-u^{r,\varepsilon}(s)\|^2_H>\delta \Big)
=-\infty.
\end{eqnarray*}
We complete the proof.

\end{proof}
Now, we can conclude the following result.
\begin{prp}\label{prp-0}
For any $f \in  C^{1+l}(\mathbb{T}^d)$ for some $l>0$. Under Hypotheses H1-H3, (\ref{eq-8}) holds, which implies the large deviation principles holds.
\end{prp}
\begin{proof}
Note that
\begin{eqnarray}\notag
&&P\Big( \sup_{0\leq s\leq 1}\|u^{\varepsilon}(s)-v^{\varepsilon}(s)\|^2_H>\delta\Big)\\
\notag
&\leq& P\Big( \sup_{0\leq s\leq 1}\|u^{\varepsilon}(s)-u^{r,\varepsilon}(s)\|^2_H>\frac{\delta}{4}\Big)
+P\Big( \sup_{0\leq s\leq 1}\|u^{r,\varepsilon}(s)-u^{r,\varepsilon}_n(s)\|^2_H>\frac{\delta}{4}\Big)\\
\label{eqqq-37}
&& +P\Big( \sup_{0\leq s\leq 1}\|u^{r,\varepsilon}_n(s)-v^{\varepsilon}_n(s)\|^2_H>\frac{\delta}{4}\Big)
+P\Big( \sup_{0\leq s\leq 1}\|v^{\varepsilon}_n(s)-v^{\varepsilon}(s)\|^2_H>\frac{\delta}{4}\Big).
\end{eqnarray}
From Lemma \ref{lemm-6}, we have that for any $R>0$, there exists a $r_0$ such that
\begin{eqnarray*}
P\Big( \sup_{0\leq s\leq 1}\|u^{\varepsilon}(s)-u^{r_0,\varepsilon}(s)\|^2_H>\frac{\delta}{4}\Big)\leq e^{-\frac{R}{\varepsilon}}, \ \forall \varepsilon\in (0,1].
\end{eqnarray*}
In view of Lemma \ref{lemm-3} and Lemma \ref{lemm-4}, for such $r_0$, there exists  $N_0$ such that
\begin{eqnarray*}
P\Big( \sup_{0\leq s\leq 1}\|u^{r_0,\varepsilon}(s)-u^{r_0,\varepsilon}_{N_0}(s)\|^2_H>\frac{\delta}{4}\Big)\leq e^{-\frac{R}{\varepsilon}}, \ \forall \varepsilon\in (0,1],
\end{eqnarray*}
and
\begin{eqnarray*}
P\Big( \sup_{0\leq s\leq 1}\|v^{\varepsilon}_{N_0}(s)-v^{\varepsilon}(s)\|^2_H>\frac{\delta}{4}\Big)\leq e^{-\frac{R}{\varepsilon}}, \ \forall \varepsilon\in (0,1].
\end{eqnarray*}
From Lemma \ref{lemm-5}, for such $r_0$ and $N_0$, there exists  $\varepsilon_0$ such that for any $\varepsilon\in (0,\varepsilon_0]$,
\begin{eqnarray*}
P\Big( \sup_{0\leq s\leq 1}\|u^{r_0,\varepsilon}_{N_0}(s)-v^{\varepsilon}_{N_0}(s)\|^2_H>\frac{\delta}{4}\Big)\leq e^{-\frac{R}{\varepsilon}}.
\end{eqnarray*}
Thus, combing all the previous estimates and by (\ref{eqqq-37}), we have that for any $\varepsilon\in (0,\varepsilon_0]$,
\begin{eqnarray*}
P\Big( \sup_{0\leq s\leq 1}\|u^{\varepsilon}(s)-v^{\varepsilon}(s)\|^2_H>\delta\Big)\leq 4e^{-\frac{R}{\varepsilon}}.
\end{eqnarray*}
Since $R$ is arbitrary, we conclude that
\begin{eqnarray*}
\lim_{\varepsilon\rightarrow 0}\varepsilon \log P\Big( \sup_{0\leq s\leq 1}\|u^{\varepsilon}(s)-v^{\varepsilon}(s)\|^2_H>\delta\Big)=-\infty,
\end{eqnarray*}
which implies (\ref{eq-8}) holds.
\end{proof}

\section{Proof of (\ref{eq-8}) without the stronger condition }\label{section-2}
In this part, we devote to relaxing the condition imposed on the initial value $f$. It reads as follows.
\begin{prp}\label{prp-2}
Let the initial value $f \in  L^p(\mathbb{T}^d)$ for all $p\in [1, \infty)$.
Under Hypotheses H1-H3, (\ref{eq-8}) holds.
\end{prp}

Before giving the proof of Proposition \ref{prp-2}, we firstly establish a lemma.

 \

 For any initial value $f \in  L^p(\mathbb{T}^d)$ for all $p\in [1, \infty)$, we denote by $u^{\varepsilon}$ the solution of (\ref{eqq-5}) with respect to the initial value $f$ and diffusion coefficient $\sigma$ satisfying Hypotheses H1-H3.
Let $\{f_k\}_{k\geq 1}\subset C^{\infty}(\mathbb{T}^d)$ be a sequence satisfies
$\|f_k-f\|^2_{L^p}\rightarrow 0$, as $k\rightarrow \infty$. Denote by $u^{\varepsilon,k}$ the solution of (\ref{eqq-5}) with respect to the initial value $f_k$ and the same diffusion coefficient $\sigma$ as above.
Under Hypotheses H1-H2, using the same method as the proof of Theorem 4.2 and Proposition 4.3 in \cite{H-Z}, we obtain
\begin{eqnarray}\label{eqqq-44-1}
\sup_{0<\varepsilon\leq 1}\Big\{E\sup_{t\in [0,1]}\|u^{\varepsilon,k}(t)\|^2_{H}+\int^1_0E\|u^{\varepsilon, k}(t)\|^2_{H^1}dt\Big\}&<&\infty,\\
\label{eqqq-45-1}
\sup_{0<\varepsilon\leq 1}E\sup_{t\in [0,1]}\|u^{\varepsilon, k}(t)\|^p_{L^p}&<&\infty.
\end{eqnarray}
Moreover, under Hypothesis H1, similar to Lemma \ref{lem-1}, we deduce that for any $p\in[1,\infty)$,
\begin{eqnarray}\label{eqqq-67-1}
\sup_{0<\varepsilon\leq 1}\Big\{E\sup_{0\leq t\leq 1}\|u^{\varepsilon}(t)\|^{2p}_H+E\int^1_0\|u^{\varepsilon}(t)\|^{2(p-1)}_H \|u^{\varepsilon}(t)\|^2_{H^1}dt\Big\}<\infty,
\end{eqnarray}
and
\begin{eqnarray}\label{eqqq-58-1}
\sup_{0<\varepsilon\leq 1}E\Big(\int^1_0\|u^{\varepsilon}(t)\|^2_{H^1}dt\Big)^p<\infty.
\end{eqnarray}
We claim that
\begin{lemma}\label{lem-2}
\begin{eqnarray}
\lim_{k\rightarrow \infty}\sup_{0<\varepsilon\leq 1}\|u^{\varepsilon,k}-u^{\varepsilon}\|_{L^1(\Omega\times [0,1]\times \mathbb{T}^d)}=0.
\end{eqnarray}
\end{lemma}
\begin{proof}
The proof is based on a suitable approximation of $L^1$ norm.
{\color{rr}
Let $1> a_1>a_2>\cdot\cdot\cdot>a_m>\cdot\cdot\cdot>0$ be a fixed sequence of decreasing positive numbers such that
\[
\int^{1}_{a_1}\frac{1}{r}dr=1,\ \cdot\cdot\cdot , \int^{a_{m-1}}_{a_m}\frac{1}{r}dr=m, \cdot\cdot\cdot
\]
Let $\psi_m(r)$ be a continuous function such that $supp(\psi_m)\subset (a_m, a_{m-1})$ and
\[
0\leq \psi_m(r)\leq 2\frac{1}{m}\times \frac{1}{r},\quad  \int^{a_{m-1}}_{a_m}\psi_m(r)dr=1.
\]
Define
\[
\phi_m(x)=\int^{|x|}_0\int^y_0\psi_m(r)drdy, \quad {\rm{for}}\ x\in \mathbb{R}.
\]
We have
\begin{eqnarray}\label{equ-00}
  \ |\phi'_m(x)|\leq 1,\ 0\leq\phi''_m(x)\leq 2\frac{1}{m}\times \frac{1}{|x|},
\end{eqnarray}
and
\begin{eqnarray}\label{equ-46}
\phi_m(x)\rightarrow |x|, \quad {\rm{as}}\ m\rightarrow\infty.
\end{eqnarray}}
Define a functional $\Phi_m:H\rightarrow \mathbb{R}$ by
\[
\Phi_m(\gamma)=\int_{\mathbb{T}^d}\phi_m(\gamma(z))dz,\quad \gamma\in H.
\]
Then, we have
\begin{eqnarray*}
\Phi'_m(\gamma)(h)&=&\int_{\mathbb{T}^d}\phi'_m(\gamma(z))h(z)dz.
\end{eqnarray*}
and
\begin{eqnarray*}
\Phi''_m(\gamma)(h,g)&=&\int_{\mathbb{T}^d}\phi''_m(\gamma(z))h(z)g(z)dz.
\end{eqnarray*}

From (\ref{eqq-5}), we obtain
\begin{eqnarray*}
&&u^{\varepsilon,k}(t)-u^{\varepsilon}(t)+\varepsilon \int^t_0 div(B(u^{\varepsilon,k}(s))-B(u^{\varepsilon}(s)))ds\\
&=& f_k-f+\varepsilon \int^t_0 div(A(u^{\varepsilon,k}(s)\nabla u^{\varepsilon,k}(s))-A(u^{\varepsilon}(s))\nabla u^{\varepsilon}(s) )ds\\
&&\ +\sqrt{\varepsilon} \int^t_0 (\sigma(u^{\varepsilon,k}(s) )-\sigma(u^{\varepsilon}(s)))dW(s).
\end{eqnarray*}
Applying the chain rule, we obtain
\begin{eqnarray*}
&&\Phi_m(u^{\varepsilon,k}(t)-u^{\varepsilon}(t))\\
&=& \Phi_m(f_k-f)+\int^t_0\Phi'_m(u^{\varepsilon,k}(s)-u^{\varepsilon}(s))d(u^{\varepsilon,k}(s)-u^{\varepsilon}(s))\\ \notag
&=& \Phi_m(f_k-f)+\varepsilon\int^t_0\int_{\mathbb{T}^d}\phi'_m(u^{\varepsilon,k}(s,z)-u^{\varepsilon}(s,z))\Big(-div(B(u^{\varepsilon,k}(s,z)))+div(B(u^{\varepsilon}(s,z)))\Big)dzds\\ \notag
&& +\varepsilon\int^t_0\int_{\mathbb{T}^d}\phi'_m(u^{\varepsilon,k}(s,z)-u^{\varepsilon}(s,z))\Big(div(A(u^{\varepsilon,k}(s,z))\nabla u^{\varepsilon,k}(s,z))-div(A(u^{\varepsilon}(s,z))\nabla u^{\varepsilon}(s,z))\Big)dzds\\ \notag
&& +\sqrt{\varepsilon}\int^t_0\int_{\mathbb{T}^d}\phi'_m(u^{\varepsilon,k}(s,z)-u^{\varepsilon}(s,z))\Big(\sigma(u^{\varepsilon,k}(s,z))-\sigma(u^{\varepsilon}(s,z))\Big)dzds\\ \notag
&& +\frac{\varepsilon}{2}\int^t_0tr[(\sigma(u^{\varepsilon,k}(s,z))-\sigma(u^{\varepsilon}(s,z)))^*\circ\Phi''_m(\sigma(u^{\varepsilon,k}(s,z))-\sigma(u^{\varepsilon}(s,z)))\circ(\sigma(u^{\varepsilon,k}(s,z))-\sigma(u^{\varepsilon}(s,z)))]ds\\ \notag
&:=&\Phi_m(f_k-f)+I^{m,\varepsilon,k}_1(t)+I^{m,\varepsilon,k}_2(t)+I^{m,\varepsilon,k}_3(t)+I^{m,\varepsilon,k}_4.
\end{eqnarray*}
Exactly argued as (3.8)-(3.10) in \cite{H-Z}, we have
\begin{eqnarray*}
I^{m,\varepsilon,k}_1(t)+I^{m,\varepsilon,k}_2(t)+I^{m,\varepsilon,k}_4(t)&\leq& \frac{C\varepsilon}{m}\int^t_0\int_{\mathbb{T}^d}|\nabla u^{\varepsilon,k}(s,z)|dzds+\frac{C\varepsilon}{m}\int^t_0\int_{\mathbb{T}^d}|\nabla u^{\varepsilon}(s,z)|dzds\\ \notag
&&+\frac{C\varepsilon}{m}\int^t_0\int_{\mathbb{T}^d}|\nabla u^{\varepsilon}(s,z)|^2dzds+\frac{C\varepsilon}{m}\int^t_0\int_{\mathbb{T}^d}|\nabla u^{\varepsilon,k}(s,z)|^2dzds\\
&& +\frac{C\varepsilon}{m}\int^t_0\int_{\mathbb{T}^d}|u^{\varepsilon,k}(s,z)-u^{\varepsilon}(s,z)|dzds.
\end{eqnarray*}
Hence, it yields
\begin{eqnarray*}
&&\int^1_0\Phi_m(u^{\varepsilon,k}(t)-u^{\varepsilon}(t))dt\\
&\leq& \Phi_m(f_k-f)+\frac{C\varepsilon}{m}\int^1_0\int_{\mathbb{T}^d}(1+|\nabla u^{\varepsilon}(s,z)|^2)dzds+\frac{C\varepsilon}{m}\int^1_0\int_{\mathbb{T}^d}(1+|\nabla u^{\varepsilon,k}(s,z)|^2)dzds\\
&& +\frac{C\varepsilon}{m}\int^1_0\int_{\mathbb{T}^d}|u^{\varepsilon,k}(s,z)-u^{\varepsilon}(s,z)|dzds+\int^1_0I^{m,\varepsilon,k}_3(t)dt.
\end{eqnarray*}
Then, taking expectation, we obtain
\begin{eqnarray*}
&&E\int^1_0\Phi_m(u^{\varepsilon,k}(t)-u^{\varepsilon}(t))dt\\
&\leq& \Phi_m(f_k-f)+\frac{C\varepsilon}{m}E\int^1_0\int_{\mathbb{T}^d}(1+|\nabla u^{\varepsilon}(s,z)|^2)dzds+\frac{C\varepsilon}{m}E\int^1_0\int_{\mathbb{T}^d}(1+|\nabla u^{\varepsilon,k}(s,z)|^2)dzds\\
&& +\frac{C\varepsilon}{m}E\int^1_0\int_{\mathbb{T}^d}|u^{\varepsilon,k}(s,z)-u^{\varepsilon}(s,z)|dzds.
\end{eqnarray*}
Taking into account (\ref{eqqq-44-1}) and (\ref{eqqq-67-1}), it follows that
\begin{eqnarray*}
E\int^1_0\Phi_m(u^{\varepsilon,k}(t)-u^{\varepsilon}(t))dt
\leq \Phi_m(f_k-f)+\frac{C\varepsilon}{m}.
\end{eqnarray*}
Letting $m\rightarrow \infty$, utilizing dominated convergence theorem and by (\ref{equ-46}), we deduce that for any $0<\varepsilon\leq 1$,
\begin{eqnarray*}
E\int^1_0|u^{\varepsilon,k}(t)-u^{\varepsilon}(t)|_{L^1}dt
\leq |f_k-f|_{L^1},
\end{eqnarray*}
which implies that
\begin{eqnarray*}
\sup_{0<\varepsilon\leq 1}E\int^1_0|u^{\varepsilon,k}(t)-u^{\varepsilon}(t)|_{L^1}dt
\leq |f_k-f|_{L^1}\rightarrow 0.
\end{eqnarray*}
Thus, we obtain the desired result.

\end{proof}

Now, we are in a position to give the proof of Proposition \ref{prp-2}.

\textbf{Proof of Proposition \ref{prp-2}}. \quad
With the help of Lemma \ref{lem-2}, by (\ref{eqqq-44-1})-(\ref{eqqq-58-1}) and utilizing Vitali's convergence theorem, we deduce that
\begin{eqnarray}\label{eqqq-46}
\sup_{0<\varepsilon\leq1}\|u^{\varepsilon, k}- u^{\varepsilon}\|_{L^p(\Omega\times [0,1]\times \mathbb{T}^d)}\rightarrow 0, \quad \forall p\in [1, \infty).
\end{eqnarray}

Applying It\^{o} formula to $\|u^{\varepsilon}(t)-u^{\varepsilon,k}(t)\|^2_H$, we have
\begin{eqnarray}\notag
\|u^{\varepsilon}(t)-u^{\varepsilon,k}(t)\|^2_H&=&\|f_k-f\|^2_H-2\varepsilon \int^t_0\langle u^{\varepsilon}(s)-u^{\varepsilon,k}(s), div(B(u^{\varepsilon}(s))-B(u^{\varepsilon,k}(s))) \rangle ds\\ \notag
&&\ +2\varepsilon \int^t_0\langle u^{\varepsilon}(s)-u^{\varepsilon,k}(s), div(A(u^{\varepsilon}(s)\nabla u^{\varepsilon}(s))-A(u^{\varepsilon,k}(s))\nabla u^{\varepsilon,k}(s) ) \rangle ds\\ \notag
&&\ +2\sqrt{\varepsilon}\int^t_0\langle u^{\varepsilon}(s)-u^{\varepsilon,k}(s),(\sigma(\varepsilon s,u^{\varepsilon}(s) )-\sigma(\varepsilon s,u^{\varepsilon,k}(s)))dW(s)\rangle\\ \notag
&&\ +\varepsilon\int^t_0\|\sigma(\varepsilon s,u^{\varepsilon}(s) )-\sigma(\varepsilon s,u^{\varepsilon,k}(s))\|^2_{\mathcal{L}_2(U,H)}ds\\ \notag
&\leq&\|f_k-f\|^2_H+2\varepsilon \int^t_0\langle \nabla (u^{\varepsilon}(s)-u^{\varepsilon,k}(s)), B(u^{\varepsilon}(s))-B(u^{\varepsilon,k}(s)) \rangle ds\\ \notag
&&\ -2\varepsilon \int^t_0\langle \nabla( u^{\varepsilon}(s)-u^{\varepsilon,k}(s)), A(u^{\varepsilon}(s)\nabla u^{\varepsilon}(s))-A(u^{\varepsilon,k}(s))\nabla u^{\varepsilon,k}(s)  \rangle ds\\ \notag
&&\ +2\sqrt{\varepsilon}\int^t_0\langle u^{\varepsilon}(s)-u^{\varepsilon,k}(s),(\sigma(\varepsilon s,u^{\varepsilon}(s) )-\sigma(\varepsilon s,u^{\varepsilon,k}(s)))dW(s)\rangle\\ \notag
&&\ +\varepsilon\int^t_0\|\sigma(\varepsilon s,u^{\varepsilon}(s) )-\sigma(\varepsilon s,u^{\varepsilon,k}(s))\|^2_{\mathcal{L}_2(U,H)}ds\\ \notag
\label{eqqq-46-1}
&:=& \|f_k-f\|^2_H+I_1(t)+I_2(t)+I_3(t)+I_4(t).
\end{eqnarray}
By Hypothesis H1 and Young's inequality, we get
\begin{eqnarray}\notag
I_1(t)&\leq & 2\varepsilon C\int^t_0\|u^{\varepsilon}(s)-u^{\varepsilon,k}(s)\|_{H^1} \|u^{\varepsilon}(s)-u^{\varepsilon,k}(s)\|_{H}ds\\
\label{eqqq-47}
&\leq & \varepsilon \varrho \int^t_0\|u^{\varepsilon}(s)-u^{\varepsilon,k}(s)\|^2_{H^1} ds+\frac{\varepsilon C}{\varrho}\int^t_0\|u^{\varepsilon}(s)-u^{\varepsilon,k}(s)\|^2_{H} ds.
\end{eqnarray}
By Hypothesis H1, we deduce that
\begin{eqnarray*}
I_2(t)
&=&-2\varepsilon \int^t_0\langle \nabla(u^{\varepsilon}(s)-u^{\varepsilon,k}(s)), A(u^{\varepsilon}(s))\nabla (u^{\varepsilon}(s)- u^{\varepsilon,k}(s))  \rangle ds\\
&&\ -2\varepsilon \int^t_0\langle \nabla(u^{\varepsilon}(s)-u^{\varepsilon,k}(s)), (A(u^{\varepsilon}(s))-A(u^{\varepsilon,k}(s)))\nabla u^{\varepsilon,k}(s)  \rangle ds\\
&\leq& -2\varepsilon \varrho\int^t_0\|u^{\varepsilon}(s)-u^{\varepsilon,k}(s)\|^2_{H^1} ds\\
&&\ +2\varepsilon \int^t_0\|u^{\varepsilon}(s)-u^{\varepsilon,k}(s)\|_{H^1}\Big(\int_{\mathbb{T}^d}|A(u^{\varepsilon}(s,x))-A(u^{\varepsilon,k}(s))|^2|\nabla u^{\varepsilon,k}(s,x)|^2dx\Big)^{\frac{1}{2}}ds\\
&\leq& -\frac{3}{2}\varepsilon \varrho\int^t_0\|u^{\varepsilon}(s)-u^{\varepsilon,k}(s)\|^2_{H^1} ds\\
&&\ +\frac{\varepsilon C}{\varrho}\int^t_0\int_{\mathbb{T}^d}|A(u^{\varepsilon}(s,x))-A(u^{\varepsilon,k}(s,x))|^2|\nabla u^{\varepsilon,k}(s,x)|^2dxds.
\end{eqnarray*}
Set
\[
N^{\varepsilon, k}(t):=\frac{\varepsilon C}{\varrho}\int^t_0\int_{\mathbb{T}^d}|A(u^{\varepsilon}(s,x))-A(u^{\varepsilon,k}(s,x))|^2|\nabla u^{\varepsilon,k}(s,x)|^2dxds.
\]
Hence, it follows that
\begin{eqnarray}\label{eqqq-48}
I_2(t)&\leq& -\frac{3}{2}\varepsilon \varrho\int^t_0\|u^{\varepsilon}(s)-u^{\varepsilon,k}(s)\|^2_{H^1} ds+N^{\varepsilon, k}(t).
\end{eqnarray}
By Hypothesis H1, it yields
\begin{eqnarray}\label{eqqq-49}
I_4(t)\leq \varepsilon C\int^t_0\|u^{\varepsilon}(s)-u^{\varepsilon,k}(s)\|^2_H ds.
\end{eqnarray}
Combing (\ref{eqqq-46-1})-(\ref{eqqq-49}), we deduce that
\begin{eqnarray*}
\|u^{\varepsilon}(t)-u^{\varepsilon,k}(t)\|^2_H&\leq & \|f_k-f\|^2_H+ N^{\varepsilon, k}(t)+\Big(\varepsilon C+\frac{\varepsilon C}{\varrho}\Big)\int^t_0\|u^{\varepsilon}(s)-u^{\varepsilon,k}(s)\|^2_{H} ds+|I_3(t)|.
\end{eqnarray*}
By Gronwall inequality, it follows that
\begin{eqnarray}\label{eqqq-51}
\|u^{\varepsilon}(t)-u^{\varepsilon,k}(t)\|^2_H&\leq&  \exp\Big\{\Big(\varepsilon C+\frac{\varepsilon C}{\varrho}\Big)t\Big\}\Big[\|f_k-f\|^2_H+N^{\varepsilon, k}(t)+|I_3(t)|\Big].
\end{eqnarray}
Using similar argument as (\ref{eqqq-2}), we deduce that
\begin{eqnarray*}
&&\Big(E\sup_{0\leq s\leq t}\|u^{\varepsilon}(s)-u^{\varepsilon,k}(s)\|^{2p}_H\Big)^{\frac{2}{p}}\\
&\leq&  \exp\Big\{\Big(\varepsilon C+\frac{\varepsilon C}{\varrho}\Big) t\Big\}\Big[\|f_k-f\|^4_H+\Big(E\sup_{0\leq s\leq t}|N^{\varepsilon, k}(s)|^p\Big)^{\frac{2}{p}}\Big]\\
&&\ +8\varepsilon C\exp\Big\{\Big(\varepsilon C+\frac{\varepsilon C}{\varrho}\Big)t\Big\}\int^t_0(E\sup_{0\leq l\leq s}\|u^{\varepsilon}(l)-u^{\varepsilon,k}(l)\|^{2p}_H)^{\frac{2}{p}}ds.
\end{eqnarray*}
Applying Gronwall inequality again, it yields that
\begin{eqnarray}\notag
&&\Big(E\sup_{0\leq s\leq 1}\|u^{\varepsilon}(s)-u^{\varepsilon,k}(s)\|^{2p}_H\Big)^{\frac{2}{p}}\\
\notag
&\leq&  \exp\Big\{\Big(\varepsilon C+\frac{\varepsilon C}{\varrho}\Big) \Big\}\Big[\|f_k-f\|^4_H+\Big(E\sup_{0\leq s\leq 1}|N^{\varepsilon, k}(s)|^p\Big)^{\frac{2}{p}}\Big]\\ \label{eqqq-52}
&&\ \times \exp\Big\{8\varepsilon C\exp\Big\{\Big(\varepsilon C+\frac{\varepsilon C}{\varrho}\Big)\Big\}\Big\}.
\end{eqnarray}
Let $p=\frac{2}{\varepsilon}$ in (\ref{eqqq-52}) and by Chebyshev inequality, we get
\begin{eqnarray*}
&&\varepsilon \log P\Big(\sup_{0\leq s\leq 1}\|u^{\varepsilon}(s)-u^{\varepsilon,k}(s)\|^{2}_H>\delta\Big)\\
&\leq& \log \Big(E\sup_{0\leq s\leq 1}\|u^{\varepsilon}(s)-u^{\varepsilon,k}(s)\|^{2p}_H\Big)^{\frac{2}{p}}-2\log \delta\\
&\leq&  \varepsilon C+\frac{\varepsilon C}{\varrho}+8\varepsilon C\exp\Big\{\Big(\varepsilon C+\frac{\varepsilon C}{\varrho}\Big) \Big\}\\
&&\ +\log\Big[\|f_k-f\|^4_H+\Big(E\sup_{0\leq s\leq 1}|N^{\varepsilon, k}(s)|^p\Big)^{\frac{2}{p}}\Big],
\end{eqnarray*}
which implies that
\begin{eqnarray}\notag
&&\sup_{0<\varepsilon \leq 1}\varepsilon \log P\Big(\sup_{0\leq s\leq 1}\|u^{\varepsilon}(s)-u^{\varepsilon,k}(s)\|^{2}_H>\delta\Big)\\ \notag
&\leq& C+\frac{ C}{\varrho}+8 C\exp\Big\{C+\frac{C}{\varrho} \Big\}\\
\label{eqqq-53}
&&\ +\log\Big[\|f_k-f\|^4_H+\sup_{0<\varepsilon\leq 1}\Big(E\sup_{0\leq s\leq 1}|N^{\varepsilon, k}(s)|^p\Big)^{\frac{2}{p}}\Big].
\end{eqnarray}
For any constant $R>0$, by (ii) in Hypothesis H1, it yields that
\begin{eqnarray*}
N^{\varepsilon, k}(t)&=&\frac{\varepsilon C}{\varrho}\int^t_0\int_{\mathbb{T}^d}|A(u^{\varepsilon}(s,x))-A(u^{\varepsilon,k}(s,x))|^2|\nabla u^{\varepsilon,k}(s,x)|^2I_{\{|\nabla u^{\varepsilon,k}(s,x)|\leq R\}}dxds\\
&&\ +\frac{\varepsilon C}{\varrho}\int^t_0\int_{\mathbb{T}^d}|A(u^{\varepsilon}(s,x))-A(u^{\varepsilon,k}(s,x))|^2|\nabla u^{\varepsilon,k}(s,x)|^2I_{\{|\nabla u^{\varepsilon,k}(s,x)|> R\}}dxds\\
&\leq& \frac{\varepsilon C}{\varrho}R^2\int^t_0\|u^{\varepsilon}(s)-u^{\varepsilon,k}(s)\|^2_Hds+\frac{\varepsilon C}{\varrho}\int^t_0\int_{\mathbb{T}^d}|\nabla u^{\varepsilon,k}(s,x)|^2I_{\{|\nabla u^{\varepsilon,k}(s,x)|> R\}}dxds.
\end{eqnarray*}
Hence,
\begin{eqnarray}\notag
\sup_{0<\varepsilon\leq 1}\Big(E\sup_{0\leq s\leq 1}|N^{\varepsilon, k}(s)|^p\Big)
&\leq & \frac{C}{\varrho^p}R^{2p}\sup_{0<\varepsilon\leq 1} \Big(E\int^1_0\|u^{\varepsilon}(s)-u^{\varepsilon,k}(s)\|^{2p}_Hds\Big)\\
\label{eqqq-55}
&&\ +\frac{C}{\varrho^p}\sup_{0<\varepsilon\leq 1} \Big( E\Big|\int^1_0\int_{\mathbb{T}^d}|\nabla u^{\varepsilon,k}(s,x)|^2I_{\{|\nabla u^{\varepsilon,k}(s,x)|> R\}}dxds\Big|^p\Big).
\end{eqnarray}
We claim that
\begin{eqnarray}\label{eqqq-59}
\lim_{k\rightarrow \infty}\sup_{0<\varepsilon\leq 1}\Big(E\sup_{0\leq s\leq 1}|N^{\varepsilon, k}(s)|^p\Big)^{\frac{2}{p}}=0.
\end{eqnarray}
Indeed,  by using the same argument as the proof of (\ref{eqqq-58}), we deduce that for any $p\geq 1$,
\begin{eqnarray}\label{eqqq-54}
\sup_{0<\varepsilon\leq 1}E\Big|\int^1_0\| u^{\varepsilon,k}(s)\|^2_{H^1}ds\Big|^p<\infty,
\end{eqnarray}
hence, by (\ref{eqqq-54}), for any $\iota>0$, there exists a constant $R_0$ such that
\begin{eqnarray}\label{eqqq-56}
\sup_{0<\varepsilon\leq 1} E\Big|\int^1_0\int_{\mathbb{T}^d}|\nabla u^{\varepsilon,k}(s,x)|^2I_{\{|\nabla u^{\varepsilon,k}(s,x)|> R_0\}}dxds\Big|^p\leq \frac{\varrho^p \iota^{\frac{p}{2}}}{2C}.
\end{eqnarray}
Taking $R=R_0$ in (\ref{eqqq-55}) and by using (\ref{eqqq-56}), we get
\begin{eqnarray}\label{eq-9}
\sup_{0<\varepsilon\leq 1}E\sup_{0\leq s\leq 1}|N^{\varepsilon, k}(s)|^p
\leq  \frac{C}{\varrho^p}R^{2p}_0 \sup_{0<\varepsilon\leq 1} E\int^1_0\|u^{\varepsilon}(s)-u^{\varepsilon,k}(s)\|^{2p}_Hds+\frac{\iota^{\frac{p}{2}}}{2}.
\end{eqnarray}
In view of (\ref{eqqq-46}), we have
\begin{eqnarray*}
\sup_{0<\varepsilon\leq 1}E\int^1_0\|u^{\varepsilon}(s)-u^{\varepsilon,k}(s)\|^{2p}_Hds\leq \sup_{0<\varepsilon\leq 1} E\int^1_0\|u^{\varepsilon}(s)-u^{\varepsilon,k}(s)\|^{2p}_{L^{2p}}ds\rightarrow 0, \quad k\rightarrow \infty,
\end{eqnarray*}
which implies that there exists $K_0$ such that for any $k\geq K_0$,
\begin{eqnarray}\label{eqqq-57}
\sup_{0<\varepsilon\leq 1}E\int^1_0\|u^{\varepsilon}(s)-u^{\varepsilon,k}(s)\|^{2p}_Hds\leq \frac{\varrho^p \iota^{\frac{p}{2}} }{2 C R^{2p}_0}.
\end{eqnarray}
Combing (\ref{eq-9}) and (\ref{eqqq-57}), we deduce that there exists $K_0$ such that for any $k\geq K_0$,
\begin{eqnarray*}
\sup_{0<\varepsilon\leq 1}E\sup_{0\leq s\leq 1}|N^{\varepsilon, k}(s)|^p
\leq  \iota^{\frac{p}{2}},
\end{eqnarray*}
thus,
\begin{eqnarray*}
\sup_{0<\varepsilon\leq 1}\Big(E\sup_{0\leq s\leq 1}|N^{\varepsilon, k}(s)|^p\Big)^{\frac{2}{p}}\leq\Big(\sup_{0<\varepsilon\leq 1}E\sup_{0\leq s\leq 1}|N^{\varepsilon, k}(s)|^p\Big)^{\frac{2}{p}}
\leq \iota.
\end{eqnarray*}
By the arbitrary of $\iota$, we complete the verification of (\ref{eqqq-59}).

Utilizing (\ref{eqqq-53}) and (\ref{eqqq-59}), it yields that
\begin{eqnarray}\label{eqqq-63}
\lim_{k\rightarrow \infty}\sup_{0<\varepsilon \leq 1}\varepsilon \log P\Big(\sup_{0\leq s\leq 1}\|u^{\varepsilon}(s)-u^{\varepsilon,k}(s)\|^{2}_H>\delta\Big)=-\infty.
\end{eqnarray}
Denote by $v^{\varepsilon,k}$ and $v^{\varepsilon}$ the solution of (\ref{eq-7}) with respect to the initial value $f_k$ and $f$, respectively.
Using similar argument as the proof of (\ref{eqqq-52}), we obtain
\begin{eqnarray*}
&&\Big(E\sup_{0\leq s\leq t}\|v^{\varepsilon}(s)-v^{\varepsilon,k}(s)\|^{2p}_H\Big)^{\frac{2}{p}}\\
&\leq& \varepsilon^2 t^2\|f_k-f\|^4_He^{\varepsilon C  t} +8\varepsilon Ce^{\varepsilon C  t}\int^t_0\Big(E\sup_{0\leq l\leq s}\|v^{\varepsilon}(l)-v^{\varepsilon,k}(l)\|^{2p}_H\Big)^{\frac{2}{p}}ds.
\end{eqnarray*}
By Gronwall inequality, it follows that
\begin{eqnarray}\notag
&&\Big(E\sup_{0\leq s\leq t}\|v^{\varepsilon}(s)-v^{\varepsilon,k}(s)\|^{2p}_H\Big)^{\frac{2}{p}}\\
\label{eqqq-73}
&\leq& \varepsilon^2 t^2e^{\varepsilon C  t}\|f_k-f\|^4_H\exp\{8\varepsilon C te^{\varepsilon C  t}\}.
\end{eqnarray}
Let $p=\frac{2}{\varepsilon}$ in (\ref{eqqq-73}) and by Chebyshev inequality, we get
\begin{eqnarray*}
&&\varepsilon \log P\Big(\sup_{0\leq s\leq 1}\|v^{\varepsilon}(s)-v^{\varepsilon,k}(s)\|^{2}_H>\delta\Big)\\
&\leq& \log \Big(E\sup_{0\leq s\leq 1}\|v^{\varepsilon}(s)-v^{\varepsilon,k}(s)\|^{2p}_H\Big)^{\frac{2}{p}}-2\log \delta\\
&\leq& 8\varepsilon C e^{\varepsilon C }+\varepsilon C+\log\Big[ \varepsilon^2 \|f_k-f\|^4_H\Big]
\end{eqnarray*}
Then, it follows that
\begin{eqnarray*}
&&\sup_{0<\varepsilon\leq 1}\varepsilon \log P\Big(\sup_{0\leq s\leq 1}\|v^{\varepsilon}(s)-v^{\varepsilon,k}(s)\|^{2}_H>\delta\Big)\\
&\leq& 8C e^{C  }+C +\log\Big[ \|f_k-f\|^4_H\Big],
\end{eqnarray*}
which yields
\begin{eqnarray}\label{eqqq-61}
\lim_{k\rightarrow \infty}\sup_{0<\varepsilon \leq 1}\varepsilon \log P\Big( \sup_{0\leq s\leq 1}\|v^{\varepsilon,k}(s)-v^{\varepsilon}(s)\|^2_H>\delta\Big)=-\infty.
\end{eqnarray}
From Proposition \ref{prp-0}, we know that
\begin{eqnarray}\label{eqqq-39}
\lim_{\varepsilon\rightarrow 0}\varepsilon \log P\Big( \sup_{0\leq s\leq 1}\|u^{\varepsilon,k}(s)-v^{\varepsilon,k}(s)\|^2_H>\delta\Big)=-\infty, \quad \forall k\geq 1.
\end{eqnarray}
Finally, with  the help of (\ref{eqqq-63}), (\ref{eqqq-61}) and (\ref{eqqq-39}), we complete the proof.

\

\noindent{\bf  Acknowledgements}\quad  The authors are grateful to the anonymous referees for comments and suggestions. This work is partly supported by National Natural Science Foundation of China (No. 11801032). Key Laboratory of Random Complex Structures and Data Science, Academy of Mathematics and Systems Science, Chinese Academy of Sciences (No. 2008DP173182). China Postdoctoral Science Foundation funded project (No. 2018M641204).

\def\refname{ References}


\begin{thebibliography}{2}
\bibitem{A-K} S. Aida, H. Kawabi: \emph{Short time asymptotics of a certain infinite dimensional diffusion process.}  Stochastic analysis and related topics, VII (Kusadasi, 1998), 77-124,
Progr. Probab., 48, Birkh\"{a}user Boston, Boston, MA, 2001.
\bibitem{A-Z} S. Aida, T. Zhang: \emph{On the small time asymptotics of diffusion processes on path groups.}
Potential Anal. 16, no. 1, 67-78 (2002).

\bibitem{B-Y} M.T. Barlow and M. Yor: \emph{Semi-martingale inequalities via the Garsia-Rodemich-Rumsey lemma, and applications to local time}. J. Funct. Anal. 49 198-229 (1982).
\bibitem{Br} Z. Brze\'{z}niak: \emph{ On stochastic convolution in Banach spaces and applications}. Stoch. Stoch. Rep. 61 (3-4) 245-295 (1997).
\bibitem{BP} Z. Brze\'{z}niak, S. Peszat: \emph{Space-time continuous solutions to SPDEs driven by a homogeneous Wiener process.} Studia Mathematica 137 (3) 261-299 (1999).


\bibitem{Daprato} G. Da Prato and J. Zabczyk: \emph{Stochastic Equations in Infinite Dimensions.} Cambridge Univ. Press, Cambridge, 1992.

\bibitem{Davis} B. Davis: \emph{On the $L^p$ norms of stochastic integrals and other martingales}. Duke Math. J. 43 697-704 (1976).
\bibitem{DHV} A. Debussche, M. Hofmanov\'{a}, and J. Vovelle: \emph{Degenerate parabolic stochastic partial differential equations: Quasilinear case}. Ann. Probab. 44, no. 3, 1916-1955 (2016).
\bibitem{DDMH} A. Debussche, S. De Moor and M. Hofmanov\'{a}: \emph{A regularity result for quasilinear stochastic partial differnential equations of parabolic type.} SIAM J.Math.Anal. 47, no.2, 1590-1614 (2015).

 \bibitem{DZ} A. Dembo, O. Zeitouni: \emph{Large deviations techniques and applications}. Jones and Bartlett, Boston, (1993).
     \bibitem{D-Z} Z. Dong, R. Zhang: \emph{On the small time asymptotics of 3D
  stochastic primitive equations}. Math. Methods Appl. Sci. 41, no. 16, 6336-6357 (2018).
     \bibitem{DZZ} Z. Dong, R. Zhang, T. Zhang: \emph{Large deviations for quasilinear parabolic stochastic partial differential equations.} Potential Anal (2019). https://doi.org/10.1007/s11118-019-09763-1.
      \bibitem{F-Z}S. Fang, T. Zhang: \emph{
On the small time behavior of Ornstein-Uhlenbeck processes with unbounded linear drifts.}
Probab. Theory Related Fields 114, no. 4, 487-504 (1999).
\bibitem{GMT} G. Gagneux, M. Madaune-Tort, \emph{Analyse math\'{e}matique de mod\'{e}les non lin\'{e}aires de l'ing\'{e}nierie p\'{e}troli\'{e}re}, Springer-Verlag, Berlin, 1996.




    \bibitem{H-R} M. Hino, J. Ram\'{i}rez: \emph{Small-time Gaussian behavior of symmetric diffusion semigroups.}
Ann. Probab. 31, no. 3, 1254-1295 (2003).
\bibitem{H-Z}M. Hofmanov\'{a}, T. Zhang: \emph{Quasilinear parabolic stochastic differential equations: existence, uniqueness.}  Stochastic Process. Appl. 127, no. 10, 3354-3371 (2017).

%
%

\bibitem{V}S.R.S. Varadhan: \emph{
Diffusion processes in a small time interval.}
Comm. Pure Appl. Math. 20 659-685 (1967).
\bibitem{X-Z} T. Xu, T. Zhang: \emph{On the small time asymptotics of the two-dimensional stochastic Navier-Stokes equations}. Ann.Inst.Henri Poincar\'{e} Probab. Stat. 45(4) 1002-1019 (2009).
\bibitem{ZTS} T. Zhang: \emph{On the small time asymptotics of diffusion processes on Hilbert spaces.}
Ann. Probab. 28, no. 2, 537-557 (2000).

\end{thebibliography}
\end{document}